\pdfoutput=1
\documentclass[english]{shinyart} 
\usepackage[utf8]{inputenc}
\usepackage{csquotes}
\usepackage[english]{babel}
\usepackage[T1]{fontenc}
\usepackage{tikz}
\usepackage{pgfplots}
\pgfplotsset{compat=newest}
\pgfplotsset{plot coordinates/math parser=false}
\usetikzlibrary{plotmarks}
\newlength\figureheight
\newlength\figurewidth
\usepackage{grffile}
\usepackage{float}
\usepackage{enumitem}
\usepackage{comment}
\usepackage{mathtools}

\usepackage{shinybib}
\bibliography{nlpdhgm_accel.bib}

\usepackage{tikz,pgfplots}
\usepgfplotslibrary{colorbrewer}
\pgfplotsset{compat=newest}
\pgfplotsset{plot coordinates/math parser=false}

\theoremstyle{definition}
\newtheorem{algorithm}{Algorithm}
\numberwithin{algorithm}{section}
\newtheorem{assumption}{Assumption}
\numberwithin{assumption}{section}

\crefname{assumption}{Assumption}{Assumptions}

\newcommand{\term}{\emph}
\newcommand{\defeq}{:=}
\newcommand{\grad}{\nabla}
\newcommand{\inv}[1]{#1^{-1}}
\newcommand{\field}[1]{\mathbb{#1}}
\newcommand{\N}{\field{N}}
\newcommand{\R}{\field{R}}
\newcommand{\extR}{\overline \R}
\newcommand{\abs}[1]{|#1|}
\newcommand{\norm}[1]{\|#1\|}
\newcommand{\iprod}[2]{\langle #1,#2\rangle}
\newcommand{\subdiff}{\partial}
\newcommand{\weakto}{\mathrel{\rightharpoonup}}
\newcommand{\downto}{\searrow}
\newcommand{\freevar}{\,\boldsymbol\cdot\,}
\newcommand{\isect}\cap

\DeclareMathOperator{\Dom}{dom}
\DeclareMathOperator{\closure}{cl}
\DeclareMathOperator{\conv}{conv}

\newcommand{\B}{\vmathbb{B}}

\def\tilde{\widetilde}

\def\NL{\textup{NL}}
\def\LIN{\textup{L}}
\def\Ynl{Y_{\NL}}
\def\Ylin{Y_{\LIN}}

\def\Pnl{P_{\NL}}

\newcommand{\setto}{\rightrightarrows}

\def\extR{\overline \R}
\def\linear{\mathcal{L}}

\newcommand{\linearLArrow}[1][]{\linear_{\triangleleft\ifx&#1&\else,\,#1\fi}}
\newcommand{\linearLArrowSpecial}[1][]{\linear^{\star}_{\triangleleft\ifx&#1&\else,\,#1\fi}}

\def\realopt#1{\widehat #1}
\def\this#1{#1^i}
\def\nexxt#1{#1^{i+1}}
\def\overnext#1{\bar #1^{i+1}}
\def\overx#1#2#3{\llbracket#1,#2\rrbracket^{#3}}

\def\realoptu{{\realopt{u}}}

\def\realoptx{{\realopt{x}}}

\def\realopty{{\realopt{y}}}

\def\nextu{\nexxt{u}}
\def\nextx{\nexxt{x}}

\def\nexty{\nexxt{y}}

\def\thisu{\this{u}}
\def\thisx{\this{x}}

\def\thisy{\this{y}}

\def\overnextx{\overnext{x}}

\def\Tau{T}
\def\TauTest{\Phi}
\def\tauTest{\phi}
\def\SigmaTest{\Psi}
\def\sigmaTest{\psi}

\def\GammaLift#1{\Xi_{#1}}

\def\Space{U}

\newcommand{\Test}{Z}
\newcommand{\Precond}{M}
\newcommand{\Step}{W}

\DeclareFontFamily{U}{mathx}{\hyphenchar\font45}
\DeclareFontShape{U}{mathx}{m}{n}{<-> mathx10}{}
\DeclareSymbolFont{mathx}{U}{mathx}{m}{n}
\DeclareMathAccent{\widebar}{0}{mathx}{"73}

\newcommand{\Penalty}{\Delta}

\def\kgrad#1{\grad K(#1)}
\def\kgradconj#1{[\grad K(#1)]^*}

\def\realoptw{\realopt{w}}
\def\neighu{\mathcal{U}}

\def\neighx{\mathcal{X}}
\def\neighy{\mathcal{Y}}

\def\bar{\widebar}

\def\metricRhoY{\rho_y}
\newcommand{\localRhoX}[1][]{\if\relax\detokenize{#1}\relax r_x\else r_{x,#1}\fi}
\newcommand{\localRhoY}[1][]{\if\relax\detokenize{#1}\relax r_y\else r_{y,#1}\fi}

\def\thetitle{Acceleration and global convergence of a first-order primal--dual method for nonconvex problems}

\title{\thetitle}
\date{2018-08-07}
\author{%
    Christian Clason\thanks{Faculty of Mathematics, University Duisburg-Essen, 45117 Essen, Germany (\email{christian.clason@uni-due.de})}
    \and
    Stanislav Mazurenko\thanks{Department of Mathematical Sciences, University of Liverpool, United Kingdom (\email{stan.mazurenko@gmail.com})}
    \and
    Tuomo Valkonen\thanks{ModeMat, Escuela Politécnica Nacional, Quito, Ecuador; \emph{previously} Department of Mathematical Sciences, University of Liverpool, United Kingdom (\email{tuomo.valkonen@iki.fi})}
}

\AtBeginDocument{
    \hypersetup{
        pdftitle = {\thetitle},
        pdfauthor = {C.~Clason and S.~Mazurenko and T.~Valkonen}
    }
}

\begin{document}

\maketitle

\begin{abstract}
    The primal--dual hybrid gradient method (PDHGM, also known as the Chambolle--Pock method) has proved very successful for convex optimization problems involving linear operators arising in image processing and inverse problems.
    In this paper, we analyze an extension to nonconvex problems that arise if the operator is nonlinear. Based on the idea of testing, we derive new step length parameter conditions for the convergence in infinite-dimensional Hilbert spaces and provide acceleration rules for suitably (locally and/or partially) monotone problems. Importantly, we prove linear convergence rates as well as global convergence in certain cases.
    We demonstrate the efficacy of these step length rules for PDE-constrained optimization problems.
\end{abstract}

\section{Introduction}
\noindent
Many optimization problems can be represented as minimizing a sum of two terms of the form
\begin{equation}
    \tag{P}
    \label{eq:initial-problem}
	\min_x~ G(x)+F(K(x))
\end{equation}
for some (extended) real-valued functionals $F$ and $G$ and a (possibly nonlinear) operator $K$.
For instance, in inverse problems, $G$ will typically be a fidelity term, measuring fit to data, and $F \circ K$ is a regularization term introduced to avoid ill-posedness and promote desired features in the solution. In imaging problems in particular, quite often total variation type regularization is used, in which case $K$ is composed of differential operators \cite{chambolle97image,bachmayr2009iterative,bredies2009tgv}.
In optimal control, $K$ frequently denotes the solution operator to partial or ordinary differential equations as a function of the control input. In this case $G$ and $F$ stand for control- and state-dependent contributions to the cost function, respectively. The function $F$ might also account for state constraints \cite{tuomov-pdex2nlpdhgm}.

Since the above applications usually involve high and possibly infinite-dimensional spaces, first-order numerical methods can provide the best trade-off between precision and computation time. This, however, depends on the exact formulation of the problem and the specific algorithm used. 
Nonsmooth first-order methods roughly divide into two classes: ones based on explicit subgradients, and ones based on proximal maps as introduced in \cite{rockafellar1976monotone}. The former can exhibit very slow convergence, while taking a step in the latter is often tantamount to solving the original problem.
As both $G$ and $F$ are often convex, introducing a dual variable $y$ and the convex conjugate $F^*$ of $F$, we can rewrite \eqref{eq:initial-problem} as
\begin{equation}
    \tag{S}
    \label{eq:initial-problem-decoupled}
	\min_{x}\max_{y}~ G(x)+\iprod{K(x)}{y}-F^*(y).
\end{equation}
Now, if we can decouple the primal and dual variables, and, instead of the proximal map of $x \mapsto G(x) + F(K(x))$, individually and efficiently compute the proximal maps $(I+\tau \subdiff G)^{-1}$ and $(I+\sigma \subdiff F^*)^{-1}$, methods based on proximal steps can be highly efficient. Based on this idea, for linear $K$ a decoupling algorithm -- now commonly known as the Chambolle--Pock method -- was suggested in \cite{pock2009mumford,chambolle2010first}. In \cite{chambolle2010first,chambolle2014ergodic} the authors proved the $O(1/N)$ convergence of an ergodic duality gap to zero and provided an $O(1/N^2)$ acceleration scheme when either the primal or dual objective is strongly convex.
In \cite{esser2010general}, the method was classified as the Primal--Dual Hybrid Gradient method, Modified (PDHGM).

However, frequently in applications, $K$ is not linear, making \eqref{eq:initial-problem} nonconvex.
This situation is the focus in the present work.
Our starting point is the extension of the PDHGM to nonlinear $K$ suggested in \cite{tuomov-nlpdhgm,tuomov-pdex2nlpdhgm}, where the authors proved local weak convergence without a rate under a metric regularity assumption. The method, called the NL-PDHGM (for ``nonlinear PDHGM''), and its ADMM variants have successfully been applied to problems in magnetic resonance imaging and PDE-constrained optimization \cite{tuomov-pdex2nlpdhgm,tuomov-nlpdhgm,benning2015preconditioned}.
We state it in \cref{alg:NL-PDHGM}, also incorporating references to the step length rules of the present work.
\begin{algorithm}[NL-PDHGM]
	\label{alg:NL-PDHGM}
	Pick a starting point $(x^0,y^0)$. Select step length parameters $\tau_i,\sigma_i,\omega_i > 0$ according to a suitable rule from one of \cref{thm:weak-convergence-nlpdhgm,thm:acceleration-nlpdhgm,thm:linear-convergence-nlpdhgm}.
        The iterate
	\begin{align*}
		\nextx & := (I + \tau_i \subdiff G)^{-1}(\thisx - \tau_i\kgradconj{\thisx}\thisy),\\
		\overnextx & :=  \nextx+\omega_i(\nextx-\thisx),\\
		\nexty & :=  (I + \sigma_{i+1} \subdiff F^*)^{-1}(\thisy + \sigma_{i+1} K(\overnextx)).
	\end{align*}
\end{algorithm}

Besides nonconvex ADMM \cite{benning2015preconditioned,wang2018global} (which is a closely related algorithm), first-order alternatives to the NL-PDHGM include iPiano \cite{ochs2014ipiano}, iPalm \cite{pock2017ipalm}, and an extension of the PDHGM to semiconvex functions \cite{mollenhoff2014primal}. The former two are inertial variants of forward--backward splitting, with iPalm further splitting the proximal step into two sub-blocks. We stress that none of these can be applied directly to \eqref{eq:initial-problem} if $F$ is nonsmooth \emph{and} $K$ is nonlinear, which is the focus of this work.
Another advantage of the approach based on the saddle point formulation \eqref{eq:initial-problem-decoupled} which moves all nonconvexity to $K$ is the following. Consider
\begin{equation}
	\label{eq:motivating-problem}
	\min_x \frac{1}{2}\norm{T(x)-z}^2 + F_0(K_0x),
\end{equation}
where $K$ is linear and $T$ nonlinear. Such problems arise, e.g., from total variation regularized nonlinear inverse problems, in which case $K_0=\grad$ and $T$ is a nonlinear forward operator \cite{tuomov-nlpdhgm}. As the function $F_0$ is typically nonsmooth (e.g., $F_0=\norm{\freevar}$ for total variation regularization), to apply a simple forward--backward scheme to this problem one would have to compute the proximal map of $F_0 \circ K_0$, which is seldom feasible. On the other hand, even if $T$ were linear, solving the dual problem instead as in \cite{beck2009fast} will not work either unless $T$ is unitary. However, we can rewrite \eqref{eq:motivating-problem} in the form \eqref{eq:initial-problem-decoupled} with $y=(y_1, y_2)$, $G \equiv 0$, $K(x) \defeq (K_0 x, T(x)-z)$, and $F^*(y) \defeq F_0^*(y_1) + \frac{1}{2}\norm{y_2}^2$. Now we only need to be able to compute $K$, $\grad K$, and the proximal map of $F_0^*$, all of which are typically easy.
Observe also how $F^*$ is strongly convex on the subspace corresponding to the nonlinear part of $K$. This will be useful for estimating convergence rates.

In \cite{tuomov-pdex2nlpdhgm}, based on small modifications to our original analysis in \cite{tuomov-nlpdhgm}, we showed that the acceleration scheme from \cite{chambolle2010first} for strongly convex problems can also be used with \cref{alg:NL-PDHGM} and nonlinear $K$ \emph{provided we stop the acceleration at some iteration}. Hence, no convergence rates could be obtained. \emph{In the present paper, based on a completely new and simplified analysis, we provide such rates and show that the acceleration does not have to be stopped}. 
To the best of our knowledge, this is the first work to prove convergence rates for a primal–dual method for nonsmooth  saddle point problems with nonlinear operators.
Our new analysis of the NL-PDHGM is based on the ``testing'' framework introduced in \cite{tuomov-proxtest} for preconditioned proximal point methods. In particular, we relax the metric regularity required in \cite{tuomov-nlpdhgm} to mere monotonicity \emph{at} a solution together with a three-point growth condition on $K$ around this solution. Both are essentially ``nonsmooth'' formulations of standard second-order growth conditions. 
We prove weak convergence to a critical point as well as $O(1/N^2)$ convergence (which is even global in some situations) with an acceleration rule if $\subdiff G$ or $\kgradconj{x}y$ is strongly monotone \emph{at} a primal critical point $\realoptx$.
If $\subdiff F^*$ is also strongly monotone at a dual critical point $\realopty$, we present step length rules that lead to linear convergence.
We emphasize that \emph{all the time we allow $K$ to be nonlinear, and through this the problem \eqref{eq:initial-problem} to be globally nonconvex.}  In addition, our local monotonicity assumptions are comparable nonsmooth counterparts to standard $C^2$ and positive Hessian assumptions in smooth nonconvex optimization. 

\bigskip

This work is organized as follows. We summarize the ``testing'' framework introduced in \cite{tuomov-proxtest} for preconditioned proximal point methods in \cref{sec:problem}.
We state our main results in \cref{sec:nlpdhgm-analysis}. Since block-coordinate methods have been receiving more and more attention lately --  including in the primal--dual algorithm designed in \cite{tuomov-blockcp} based on the same testing framework -- the main technical derivations of \cref{sec:testing-estimates} are implemented in a generalized operator form.
Once we have obtained these generic estimates, we devote \cref{sec:scalar} to scalar step length parameters and formulate our main convergence results. These amount to basically standard step length rules for the PDHGM combined with bounds on the initial step lengths.
Finally, in \cref{sec:examples}, we illustrate our theoretical results with numerical evidence. We study parameter identification with $L^1$ fitting and optimal control with state constraints, where the nonlinear operator $K$ involves the mapping from a potential term in an elliptic partial differential equation to the corresponding solution.

\section{Problem formulation}
\label{sec:problem}

Throughout this paper, we write $\linear(X; Y)$ for the space of bounded linear operators between Hilbert spaces $X$ and $Y$. We write $I$ for the identity operator, $\iprod{x}{x'}$ for the inner product, and $\B(x,r)$ for the closed unit ball of the radius $r$ at $x$ in the corresponding space. We set $\iprod{x}{x'}_T\defeq\iprod{Tx}{x'}$ and $\norm{x}_T\defeq\sqrt{\iprod{x}{x}_T}$. For $T,S\in\linear(X; Y)$, the inequality $T\ge S$ means $T-S$ is positive semidefinite. Finally, $\overx{x_1}{x_2}{\alpha} \defeq (1-\alpha)x_1 + \alpha x_2$; in particular, $\overnextx \defeq \overx{\nextx}{\thisx}{-\omega_i}$ in \cref{alg:NL-PDHGM}.

We generally assume $G: X \to \extR$ and $F^* \to \extR$ to be convex, proper, and lower semicontinuous, so that their subgradients $\subdiff G$ and $\subdiff F^*$ are well-defined maximally monotone operators \cite[Theorem 20.25]{bauschke2017convex}. Under a constraint qualification, e.g., when $K$ is $C^1$ and either the null space of $\kgradconj{x}$ is trivial or $\Dom F=X$ \cite[Example 10.8]{rockafellar-wets-va}, the critical point conditions for \eqref{eq:initial-problem} and \eqref{eq:initial-problem-decoupled} can be written as $0 \in H(\realoptu)$ for the 
set-valued operator $H: X \times Y \setto X \times Y$,
\begin{equation}
	\label{eq:h}
	H(u) \defeq
	\begin{pmatrix}
		\subdiff G(x) + \kgradconj{x} y \\
        \subdiff F^*(y) - K(x)
	\end{pmatrix},
\end{equation}
and $u=(x,y)\in X\times Y$.
Throughout the paper, $\realoptu\defeq(\realoptx,\realopty)$ always denotes an arbitrary root $H$, which can equivalently be characterized as $\realoptu \in \inv H(0)$.

To formulate \cref{alg:NL-PDHGM} in terms suitable for the testing framework of \cite{tuomov-proxtest}, we define the step length and testing operator
\[
    \Step_{i+1} \defeq
    \begin{pmatrix}
        \Tau_i & 0 \\
        0 & \Sigma_{i+1}
    \end{pmatrix}
    \quad\text{and}\quad
    \Test_{i+1} \defeq
    \begin{pmatrix}
        \TauTest_i & 0 \\
        0 & \SigmaTest_{i+1}
    \end{pmatrix},
\]
respectively, where $\Tau_i, \TauTest_i \in \linear(X; X)$ and $\Sigma_{i+1}, \SigmaTest_{i+1} \in \linear(Y; Y)$ are the primal step length and testing operators as well as their dual counterparts.

We also define the nonlinear preconditioner $\Precond_{i+1}$ and the partial linearization $\tilde H_{i+1}$ of $H$ by
\begin{align}
	\label{eq:precond}
	\Precond_{i+1} & \defeq
	\begin{pmatrix}
	I & -\Tau_i \kgradconj{\thisx} \\
	-\omega_i \Sigma_{i+1} \kgrad{\thisx} & I
	\end{pmatrix},
	\quad\text{and}
	\\
	\label{eq:tilde-h}
	\tilde H_{i+1}(u) & \defeq \begin{pmatrix}
	\subdiff G(x) + \kgradconj{\thisx} y \\
	\subdiff F^*(y)-K(\overx{x}{\thisx}{-\omega_i})-\kgrad{\thisx}(x-\overx{x}{\thisx}{-\omega_i})
	\end{pmatrix}.
\end{align}
Note that $\tilde H_{i+1}(u)$ simplifies to $H(u)$ for linear $K$.
Now \cref{alg:NL-PDHGM} (which coincides with the ``exact'' NL-PDHGM of \cite{tuomov-nlpdhgm}) can be written as
\begin{equation}
	\label{eq:ppext}
	\tag{PP}
	0 \in \Step_{i+1} \tilde H_{i+1}(\nextu)+\Precond_{i+1}(\nextu-\thisu).
\end{equation}
(For the ``linearized'' NL-PDHGM of \cite{tuomov-nlpdhgm}, we would replace $\overx{x}{\thisx}{-\omega}$ in \eqref{eq:tilde-h} by $\thisx$.)
Following \cite{tuomov-proxtest}, the step length operator $\Step_{i+1}$ in \eqref{eq:ppext} acts on $\tilde H_{i+1}$ rather than on the step $\nextu-\thisu$ so as to eventually allow zero-length steps on sub-blocks of variables as employed in \cite{tuomov-blockcp}. The testing operator $\Test_{i+1}$ does not yet appear in \eqref{eq:ppext} as it does not feature in the algorithm. We will shortly see that when we apply it to \eqref{eq:ppext}, the product $\Test_{i+1}\Precond_{i+1}$ will form a metric (in the differential-geometric sense) that encodes convergence rates.

Finally, we will also make use of the (possibly empty) subspace $\Ynl$ of $Y$ in which $K$ acts linearly, i.e.,
\[
	\Ylin \defeq \{ y \in Y \mid \text{the mapping } x \mapsto \iprod{y}{K(x)} \text{ is linear} \}
	\quad\text{and}\quad
	\Ynl \defeq \Ylin^\perp.
\]
(For examples of such subspaces, we refer to the introduction or, in particular, to \cite{tuomov-nlpdhgm}.) 
Furthermore, $\Pnl$ will denote the orthogonal projection to $\Ynl$. We also write $\B_\NL(\realopty, r) \defeq \{ y \in Y \mid \norm{y-\realopty}_{\Pnl} \le r\}$ for a closed cylinder in $Y$ of the radius $r$ with axis orthogonal to $\Ynl$.

\bigskip

Our goal in the rest of the paper is to analyze the convergence of \eqref{eq:ppext} for the choices \eqref{eq:h}--\eqref{eq:tilde-h}. We will base this analysis on the following abstract ``meta-theorem'', which formalizes common steps in convergence proofs of optimization methods. Its purpose is to reduce the proof of convergence to showing that the ``iteration gaps'' $\Penalty_{i+1}$ -- which encode differences in function values and whose specific form depend on the details of the algorithm -- are non-positive. 
The proof of the meta-theorem itself is relatively trivial, being based on telescoping and Pythagoras' (three-point) formula.
\begin{theorem}[{\cite[Theorem 2.1]{tuomov-proxtest}}]
    \label{thm:convergence-result-main-h}
	Suppose \eqref{eq:ppext} is solvable, and denote the iterates by $\{\thisu\}_{i \in \N}$. If $\Test_{i+1}\Precond_{i+1}$ is self-adjoint, and for some $\Penalty_{i+1} \in \R$ we have
	\begin{equation}
		\label{eq:convergence-fundamental-condition-iter-h}
		\tag{CI}
        \iprod{\tilde H_{i+1}(\nextu)}{\nextu-\realoptu}_{\Test_{i+1}}
		\ge 
		\frac{1}{2}\norm{\nextu-\realoptu}_{\Test_{i+2}\Precond_{i+2}-\Test_{i+1}\Precond_{i+1}}^2
		-\frac{1}{2}\norm{\nextu-\thisu}_{\Test_{i+1} \Precond_{i+1}}^2
		- \Penalty_{i+1}
	\end{equation}
    for all $i\le N-1$ and some $\realoptu \in \Space$, then
    \begin{equation}
        \label{eq:convergence-result-main-h}
        \tag{DI}
        \frac{1}{2}\norm{u^N-\realoptu}^2_{\Test_{N+1}\Precond_{N+1}}
        \le
        \frac{1}{2}\norm{u^0-\realoptu}^2_{\Test_{1}\Precond_{1}}
        +
        \sum_{i=0}^{N-1} \Penalty_{i+1}.
    \end{equation}
\end{theorem}
Note that the theorem always holds for \emph{some} choice of the $\Penalty_{i+1} \in \R$. Our goal will be to choose the step length and testing operators $\Tau_i,\Sigma_{i+1},\TauTest_i$ and $\Sigma_{i+1}$ as well as the over-relaxation parameter $\omega_i$ such that $\Penalty_{i+1} \le 0$ and -- in order to obtain rates -- $\Test_{i+1}\Precond_{i+1}$ grows fast as $i \to \infty$. For example, if $\Penalty_{i+1}\leq 0$ and $\Test_{N+1}\Precond_{N+1}\ge\mu_NI$ with $\mu_N\rightarrow\infty$, then clearly $\norm{u^N-\realoptu}^2\rightarrow0$ at the rate $O(1/\mu_N)$.
In other contexts, $\Delta_{i+1}$ can be used to encode duality gaps \cite{tuomov-proxtest} or a penalty on convergence rates due to inexact, stochastic, updates of the local metric $\Test_{i+1}\Precond_{i+1}$ \cite{tuomov-blockcp}.

To motivate the following, consider the ``generalized descent inequality'' \eqref{eq:convergence-fundamental-condition-iter-h} in the simple case $\tilde H_{i+1}=H$. If we now had \emph{at $\realoptu$ for $\realoptw \defeq 0 \in H(\realoptu)$} the ``operator-relative strong monotonicity''
\[
    \iprod{H(\nextu)-\realoptw}{\nextu-\realoptu}_{\Test_{i+1}\Step_{i+1}} \ge \norm{\nextu-\realoptu}^2_{\Test_{i+1}\Gamma_{i+1}}
\] 
for some suitable operator $\Gamma_{i+1}$, then the local metrics should ideally be updated as $\Test_{i+1}\Precond_{i+2}=\Test_{i+1}(\Precond_{i+1}+2\Gamma_{i+1})$. Part of our work in the following sections is to find such a $\Gamma_{i+1}$ while maintaining self-adjointness and obtaining fast growth of the metrics.
However, our specific choices of $\tilde H_{i+1}$ and $\Precond_{i+1}$ switch parts of $H$ to take the gradient step $-[\grad K(\thisx)]^*\thisy$ in the primal update and an over-relaxed step in the dual update. We will approximately undo these changes using the term $-\frac{1}{2}\norm{\nextu-\thisu}_{\Test_{i+1} \Precond_{i+1}}^2$ in \eqref{eq:convergence-fundamental-condition-iter-h}. This component of \eqref{eq:convergence-fundamental-condition-iter-h} can also be related to the ``three-point hypomonotonicity'' $\iprod{\grad G(\thisx)-\grad G(\realoptx)}{\nextx-\realoptx} \ge -\frac{L}{4}\norm{\nextx-\thisx}^2$ that holds for convex $G$ with an $L$-Lipschitz gradient \cite{tuomov-proxtest}.

Before proceeding with deriving convergence rates using this approach, we show that we can still obtain weak convergence even if $\Test_{N+1}\Precond_{N+1}$ does not grow quickly.
\begin{proposition}[weak convergence]
    \label{prop:rateless-varying}
    Suppose the iterates of \eqref{eq:ppext} satisfy \eqref{eq:convergence-fundamental-condition-iter-h} for some $\realoptu \in \inv H(0)$ with $\Test_{i+1}\Precond_{i+1}$ self-adjoint and $\Penalty_{i+1} \le -\frac{\hat{\delta}}{2}\norm{\nextu-\thisu}_{\Test_{i+1}\Precond_{i+1}}^2$ for some $\hat{\delta} > 0$.
    Assume that
    \begin{enumerate}[label=(\roman*)]
        \item\label{item:weak-zm-opers} $\varepsilon I \le \Test_{i+1}\Precond_{i+1}$ for some $\varepsilon>0$;
        \item\label{item:weak-zm-h-limit}
        for some nonsingular $\Step \in \linear(\Space; \Space)$,
        \begin{equation*}
            \Test_{i+1}\Precond_{i+1}(\nextu-\thisu) \to 0,
            \quad
            u^{i_k} \weakto \bar{u}
            \implies
            0 \in \Step H(\bar{u});
        \end{equation*}
        \item\label{item:weak-zm-a-limit} there exists a constant $C$ such that $\norm{\Test_{i}\Precond_{i}}\le C^2$ for all $i$, and for any subsequence $u^{i_k} \weakto u$ there exists $A_\infty \in \linear(\Space; \Space)$ such that $\Test_{i_k+1}\Precond_{i_k+1} u \to A_\infty u$ strongly in $U$ for all $u \in U$.
    \end{enumerate}
    Then $\thisu \weakto \bar{u}$ weakly in $\Space$ for some $\bar{u} \in \inv H(0)$.
\end{proposition}
\begin{proof}
    This is an improvement of \cite[Proposition 2.5]{tuomov-proxtest} that permits nonconstant $\Test_{i+1}\Precond_{i+1}$ and a nonconvex solution set.
    The proof is based on the corresponding improvement of Opial's lemma (\cref{lemma:opial-improved}) together with \cref{thm:convergence-result-main-h}.
    Using $\Penalty_{i+1} \le -\frac{\hat{\delta}}{2}\norm{\nextu-\thisu}_{\Test_{i+1}\Precond_{i+1}}^2$, \eqref{eq:convergence-result-main-h} applied with  $N=1$ and $\thisu$ in place of $u^0$ shows that $i \mapsto \norm{\thisu-\realoptu}_{\Test_{i+1}\Precond_{i+1}}^2$ is nonincreasing.
    This verifies \cref{lemma:opial-improved}\,\ref{item:opial-improved-non-increasing}.
    Further use of  \eqref{eq:convergence-result-main-h} shows that $\sum_{i=0}^{\infty} \frac{\hat{\delta}}{2}\norm{\nextu-\thisu}_{\Test_{i+1}\Precond_{i+1}}^2 < \infty$.
    Thus $\Test_{i+1}\Precond_{i+1}(\nextu-\thisu) \to 0$.
    By \eqref{eq:ppext} and \ref{item:weak-zm-h-limit}, any weak limit point $\bar{u}$ of the $\{\thisu\}_{i \in \N}$ therefore satisfies $\bar{u} \in \inv H(0)$.
    This verifies \cref{lemma:opial-improved}\,\ref{item:opial-improved-limit} with $\hat X=\inv H(0)$.
    The remaining assumptions of \cref{lemma:opial-improved} are verified by conditions \ref{item:weak-zm-opers} and \ref{item:weak-zm-a-limit}, which yields the claim.
\end{proof}

\section{Abstract analysis of the NL-PDHGM}
\label{sec:nlpdhgm-analysis}

We will apply \cref{thm:convergence-result-main-h} to \cref{alg:NL-PDHGM}, for which we have to verify \eqref{eq:convergence-fundamental-condition-iter-h}. This inequality always holds for some $\Penalty_{i+1}$, but for obvious reasons we aim for $\Penalty_{i+1} \le 0$.
To obtain fast convergence rates, our second goal is to make the metric $\Test_{i+1}\Precond_{i+1}$ grow as quickly as possible; the rate of this growth is constrained through \eqref{eq:convergence-fundamental-condition-iter-h} by the term $\frac{1}{2}\norm{\nextu-\realoptu}_{\Test_{i+1}\Precond_{i+1}-\Test_{i+2}\Precond_{i+2}}^2$. In this section, we therefore reduce \eqref{eq:convergence-fundamental-condition-iter-h} into a few simple conditions on the step length and testing operators.  
After stating our fundamental assumptions in \cref{sec:fundamental-assumptions}, we first derive in \cref{sec:testing-estimates} explicit (albeit somewhat technical) bounds on the step length operators to ensure \eqref{eq:convergence-fundamental-condition-iter-h}.
These require that the iterates $\{\thisu\}_{i \in \N}$ stay in a neighborhood of the critical point $\realoptu$. Therefore, in \cref{sec:locality}, we provide sufficient conditions for this requirement to hold in the form of additional step length bounds. We will use these  conditions in \cref{sec:scalar}, where we will derive the actual convergence rates for scalar step lengths.

\subsection{Fundamental assumptions}
\label{sec:fundamental-assumptions}

In what follows, we will need $K$ to be locally Lipschitz differentiable.
\begin{assumption}[locally Lipschitz $\grad K$]
	\label{ass:k-lipschitz}
        The operator $K:X\to Y$ is Fréchet differentiable, and for some $L \ge 0$ and a neighborhood $\neighx_K$ of $\realoptx$, 
    \begin{equation}
        \label{eq:ass-k-lipschitz}
        \|\kgrad{x}-\kgrad{x'}\| \le L\|x-x'\| \quad (x,x'\in\neighx_K).
    \end{equation}
\end{assumption}

\begin{remark}
	Using \cref{ass:k-lipschitz} and the mean value equality
	\[
	   K(x')=K(x)+\kgrad{x}(x'-x)+\int_{0}^{1}(\kgrad{x+s(x'-x)}-\kgrad{x})(x'-x)ds,
	\]
	we obtain for any $x,x' \in \neighx_K$ and $y\in \Dom F^*$ the useful inequality 
	\begin{equation}
		\label{eq:ass-k-lipschitz-2}
		\iprod{K(x')-K(x)-\kgrad{x}(x'-x)}{y} \le (L/2) \norm{x-x'}^2\norm{y}_{\Pnl},
	\end{equation}
	where the norm in the dual space consists of only the $\Ynl$ component because by definition, the function $x \mapsto \iprod{K(x)}{y}$ is linear in $x$ for $y\in\Ylin$. Consequently, for such $y$, the left-hand side of \cref{eq:ass-k-lipschitz-2} is zero.
\end{remark}

We also require a form of ``local operator-relative strong monotonicity'' of the saddle-point mapping $H$.
    Let $U$ be a Hilbert space, and $\Gamma \in \linear(U; U)$, $\Gamma\ge0$.
    We say that the set-valued map $H: U \setto U$ is \term{$\Gamma$-strongly monotone at $\realoptu$ for $\realoptw \in H(\realoptu)$} if there exists a neighborhood $\neighu \ni \realoptu$ such that 
    \begin{equation}
        \label{eq:monot}
        \iprod{w-\realoptw}{u-\realoptu} \ge \norm{u-\realoptu}_{\Gamma}^2, \qquad (u\in \neighu, w\in H(u)).
    \end{equation}
    If $\Gamma=0$, we say that $H$ is \term{monotone at $\realoptu$ for $\realoptw$}.

In particular, we will assume this monotonicity in terms of $\subdiff G$ and $\subdiff F^*$. The idea is that $G$ and $F^*$ can have different level of strong convexity on sub-blocks of the variables $x$ and $y$; we will in particular use this approach to assume strong convexity from $F^*$ on the subspace $\Ynl$ only. We will first need the following assumption in \cref{lemma:nonlinear-preconditioner-estimate}.
\begin{assumption}[monotone $\subdiff G$ and $\subdiff F^*$]
\label{ass:gf}
	The set-valued map $\subdiff G$ is ($\Gamma_G$-strongly) monotone at $\realoptx$ for $-\kgradconj{\realoptx}\realopty$ in the neighborhood $\neighx_G$ of $\realoptx$, and the set-valued map $\subdiff F^*$ is ($\Gamma_{F^*}$-strongly) monotone at $\realopty$ for $K(\realoptx)$ in the neighborhood $\neighy_{F^*}$ of $\realopty$.
\end{assumption}
Of course, in view of the assumed convexity of $G$ and $F^*$, \cref{ass:gf} is always satisfied with $\Gamma_G=\Gamma_{F^*}=0$.

Our next three-point assumption on $K$ is central to our analysis. It combines a second-order growth condition with a smoothness estimate, and the operator $\tilde{\Gamma}_G$ that we now introduce will later be employed as an acceleration factor. 
\begin{assumption}[three-point condition on $K$]
	\label{ass:k-nonlinear}
	For given $\Gamma_G,\tilde{\Gamma}_G\in \linear(X; X)$, neighborhood $\neighx_K$ of $\realoptx$, and some $\Lambda \in \linear(X; X)$, $\theta\ge0$, and $p \in [1,2]$ we have
	\begin{multline}
		\label{eq:ass-k-nonlinear}
		\iprod{[\kgrad{x'}-\kgrad{\realoptx}]^*\realopty}{x-\realoptx}+\norm{x-\realoptx}^2_{\Gamma_G-\tilde\Gamma_G}
		\\
		\ge \theta\norm{K(\realoptx)-K(x)-\kgrad{x}(\realoptx-x)}^p-\frac{1}{2}\norm{x-x'}^2_{\Lambda},
		\quad (x,x'\in\neighx_K).
	\end{multline}
\end{assumption}
We typically have that $0 \le \tilde \Gamma_G \le \Gamma_G$.
For linear $K$, \cref{ass:k-nonlinear} trivially holds for any $\tilde \Gamma_G \le \Gamma_G$, $\Lambda=0$, and $\theta \ge 0$. To motivate the assumption in nontrivial cases, consider the following example.
\begin{example}
	\label{ex:0d-dual}
	Let $F^*=\delta_{\{1\}}$ and take $\tilde \Gamma_G=\Gamma_G$ as well as $K(x)=J(x)$ for some $J \in C^2(X)$, which corresponds to the problem $\min_{x \in X} G(x) + J(x)$ where $J$ is smooth and possibly nonconvex but $G$ can be nonsmooth. In this case, both the over-relaxation step and dual update of \cref{alg:NL-PDHGM} are superfluous, and the entire algorithm reduces to conventional forward--backward splitting.
	If $x$ and $x'$ are suitably close to $\realoptx$, Taylor expansion shows that \eqref{eq:ass-k-nonlinear} can be expressed as
	\begin{equation}
		\label{eq:ass-k-nonlinear-jex-smooth}
		\iprod{x'-\realoptx}{x-\realoptx}_{\grad^2 J(\tilde{x}')}
		\ge
		\theta \norm{x-\realoptx}_{\grad ^2 J(\tilde{x})}^{2p}-\frac{1}{2}\norm{x-x'}^2_{\Lambda}
	\end{equation}
	for some $\tilde{x}=\tilde{x}(x,\realoptx),\tilde{x}'=\tilde{x}'(x',\realoptx) \in X$. If $\grad^2 J(\tilde{x}')$ is positive definite, i.e. $\grad^2 J(\tilde{x}')\ge\varepsilon I$ for some $\varepsilon>0$, then writing
	\begin{equation}
		\label{eq:0d-dual}
                \begin{aligned}[t]
                    \iprod{x'-\realoptx}{x-\realoptx}_{\grad^2 J(\tilde{x}')}&=\norm{x-\realoptx}_{\grad^2 J(\tilde{x}')}^2+\iprod{x'-x}{x-\realoptx}_{\grad^2 J(\tilde{x}')}\\
                    &\ge\norm{x-\realoptx}_{\grad^2 J(\tilde{x}')}^2-(1-\alpha)\norm{x-\realoptx}_{\grad^2 J(\tilde{x}')}^2\\
                    \MoveEqLeft[-1]-\frac{1}{4(1-\alpha)}\norm{x'-x}_{\grad^2 J(\tilde{x}')}^2
                    \\
                    &\ge
                    \alpha\varepsilon\norm{x-\realoptx}^2- \frac{L}{4(1-\alpha)}\norm{x'-x}^2,
                \end{aligned}
	\end{equation}
	 we see that \eqref{eq:ass-k-nonlinear-jex-smooth} holds in some neighborhood $\neighx_K$ of $\realoptx$, for any $p\in[1,2]$, $\theta>0$ small enough, and $\Lambda>0$ large enough.
	The positivity of $\grad^2 J(\tilde{x}')$ is guaranteed by the positivity of $\grad^2 J(\realoptx)$ for $x'$ close to $\realoptx$. Alternatively, recalling the full expression \eqref{eq:ass-k-nonlinear}, we can use the strong monotonicity of $\subdiff G$ at $\realoptx$. Overall, we therefore require $\Gamma_G+\grad^2 J(\realoptx)$ to be positive, which is a standard condition in nonconvex optimization.
\end{example}

If $\Dom F^*$ is not a singleton, we can apply the reasoning of \cref{ex:0d-dual} to $J(x)\defeq K(x)^*\realopty$. The positivity of $\Gamma_G+\grad^2(K(\freevar)^*\realopty)(\realoptx)$ then amounts to a second-order optimality condition on the solution $\realoptx$ to the problem $\min_x G(x)+\iprod{K(x)}{\realopty}$. 
Indeed, we can can verify \cref{ass:k-nonlinear} simply based on the monotonicity of $\subdiff G + \grad K(\freevar)^*\realopty$ at  $\realoptx$.
\begin{proposition} 
	\label{pr:strongly-convex-primal}
	Suppose \cref{ass:k-lipschitz} (locally Lipschitz $\grad K$) and \cref{ass:gf} (monotone $\subdiff G$ and $\subdiff F^*$) hold and for some $\gamma_x>0$, 
	\begin{equation}
	   \label{eq:strongly-convex-primal}
		\norm{x-\realoptx}_{\Gamma_G-\tilde\Gamma_G}^2
		+\iprod{(\kgrad{x}-\kgrad{\realoptx})(x-\realoptx)}{\realopty}
		\ge\gamma_x\norm{x-\realoptx}^2
		\quad
		(\forall x\in\neighx_K).
	\end{equation}
	Then \cref{ass:k-nonlinear} holds with $p=1$, $\theta=2(\gamma_x-\xi)L^{-1}$, and $\Lambda=L^2\norm{\Pnl \realopty}^2(2\xi)^{-1}I$ for any $\xi\in(0,\gamma_x]$.
\end{proposition}
\begin{proof}
	An application of Cauchy's inequality, \cref{ass:k-lipschitz}, and \eqref{eq:strongly-convex-primal} yields for any $\xi>0$ the estimate
	\[\begin{split}
		\iprod{[\kgrad{x'}-\kgrad{\realoptx}]^*\realopty}{x-\realoptx}+\norm{x-\realoptx}^2_{\Gamma_G-\tilde\Gamma_G}
		&=
		\iprod{[\kgrad{x}-\kgrad{\realoptx}]^*\realopty}{x-\realoptx}+\norm{x-\realoptx}^2_{\Gamma_G-\tilde\Gamma_G}
		\\
		&\quad
		+\iprod{(\kgrad{x'}-\kgrad{x})(x-\realoptx)}{\realopty}
		\\
		&
		\ge(\gamma_x-\xi)\norm{x-\realoptx}^2-L^2\norm{\Pnl \realopty}^2(4\xi)^{-1} \norm{x'-x}^2.
	\end{split}\]
	At the same time, using \cref{eq:ass-k-lipschitz} and the reasoning of \eqref{eq:ass-k-lipschitz-2},
	$\norm{K(\realoptx)-K(x)-\kgrad{x}(\realoptx-x)}\le(L/2)\norm{x-\realoptx}^2$.
	So \cref{ass:k-nonlinear} holds if we take $p=1$, $\theta\le 2(\gamma_x-\xi)/L$, and $\Lambda=L^2\norm{\Pnl \realopty}^2(2\xi)^{-1}I$.
\end{proof}

\begin{remark}
	Observe that if $\Gamma_G-\tilde\Gamma_G\ge\varepsilon I$ for some $\varepsilon>L\norm{\Pnl \realopty}$, then \cref{ass:k-lipschitz} (locally Lipschitz $\grad K$) guarantees \eqref{eq:strongly-convex-primal} for $\gamma_x=\varepsilon-L\norm{\Pnl \realopty}$. This requires $\norm{\Pnl \realopty}$ to be small, which was a central assumption in \cite{tuomov-nlpdhgm} that we intend to avoid in the present work.
    Also note that if $\iprod{K(\freevar)}{\realopty}$ is convex, then \eqref{eq:strongly-convex-primal} holds for $\gamma_x=\varepsilon$, so estimating $\gamma_x$ based on the Lipschitz continuity of $\grad K$ alone provides a too conservative estimate.
\end{remark}

More generally, while based on our discussion above the satisfaction of \eqref{ass:k-nonlinear} seems reasonable to expect, its verification can demand some effort. To demonstrate that the condition can be satisfied, we verify this in \cref{sec:complex} for a simple example of reconstructing the phase and amplitude of a complex number from a noisy measurement.

Combining \cref{ass:k-lipschitz,ass:gf,ass:k-nonlinear}, we assume throughout the rest of the paper that for some $\metricRhoY\ge0$, the corresponding neighborhood  
\begin{equation}
	\label{eq:neighu-definition}
	\neighu (\metricRhoY) \defeq (\neighx_G \isect \neighx_K) \times (\B_{\NL}(\realopty, \metricRhoY) \isect \neighy_{F^*})
\end{equation}
of $\realoptu$ is nonempty.

\subsection{General estimates}
\label{sec:testing-estimates}

We verify the conditions of \cref{thm:convergence-result-main-h} in several steps. First, we ensure that the operator $\Test_{i+1}\Precond_{i+1}$ giving rise to the local metric is self-adjoint.
Then we show that  $\Test_{i+2}\Precond_{i+2}$ and the update $\Test_{i+1}(\Precond_{i+1}+\GammaLift{i+1})$ performed by the algorithm yield identical norms, where $\GammaLift{i+1}$ represents some off-diagonal components from the algorithm as well as any strong monotonicity.
Finally, we estimate $\tilde H_{i+1}(u)$ in order to verify \eqref{eq:convergence-fundamental-condition-iter-h}.

We require for some $\kappa \in [0, 1)$, $\eta_i>0$, $\tilde\Gamma_G \in \linear(X; X)$, and $\tilde\Gamma_{F^*} \in \linear(Y;Y)$ the following relationships:
\begin{subequations}
\label{eq:basic-step-rules}
\begin{align}
    \omega_i & \defeq \eta_i/\eta_{i+1},
    & \SigmaTest_{i}\Sigma_i & = \eta_i I,
    \\
    \TauTest_i\Tau_i & = \eta_i I,
    & (1-\kappa)\SigmaTest_{i+1} & \ge \eta_i^2 \kgrad{\thisx}\inv\TauTest_i\kgradconj{\thisx},
    \\
    \TauTest_i&=\TauTest_i^* \ge 0,
    & \SigmaTest_{i+1}&=\SigmaTest_{i+1}^*\ge0,
    \\
    \label{eq:test-accel-update}
    \TauTest_{i+1}&=\TauTest_i(1+2\Tau_i\tilde\Gamma_G),
    &
    \SigmaTest_{i+2}&=\SigmaTest_{i+1}(1+2\Sigma_{i+1}\tilde\Gamma_{F^*}).
\end{align}
\end{subequations}
In \cref{sec:scalar}, we will verify these relationships for specific scalar step length rules in \cref{alg:NL-PDHGM}.

\begin{lemma}
    \label{lemma:zimi-estim}
    Fix $i \in \N$ and suppose \eqref{eq:basic-step-rules} holds. Then $\Test_{i+1}\Precond_{i+1}$ is self-adjoint and satisfies
    \[
        \Test_{i+1}\Precond_{i+1}
        \ge
        \left(\begin{smallmatrix}
            \delta \TauTest_i & 0 \\
            0 & (\kappa-\delta)(1-\delta)^{-1}\SigmaTest_{i+1}
        \end{smallmatrix}\right)
        \quad\text{ for any }\delta \in [0, \kappa].
    \]
\end{lemma}
\begin{proof}
    From \eqref{eq:precond} and \eqref{eq:basic-step-rules}, we have
    $\TauTest_i\Tau_i=\eta_i I$ and  $\SigmaTest_{i+1} \Sigma_{i+1} \omega_i=\eta_i I$. Hence
    \begin{equation}
        \label{eq:test-precond-expansion}
        \Test_{i+1}\Precond_{i+1}
        =
        \begin{pmatrix}
            \TauTest_i & -\eta_i \kgradconj{\thisx} \\
            -\eta_i \kgrad{\thisx} & \SigmaTest_{i+1}
        \end{pmatrix},
    \end{equation}
    and therefore $\Test_{i+1}\Precond_{i+1}$ is self-adjoint.
    Cauchy's inequality furthermore implies that
    \begin{equation}
    \label{eq:test-precond-expansion-estimate}
        \Test_{i+1}\Precond_{i+1}
        \ge
        \begin{pmatrix}
            \delta \TauTest_i & 0 \\
            0 & \SigmaTest_{i+1} - \frac{\eta^2_i}{1-\delta}\kgrad{\thisx}\inv\TauTest_i\kgradconj{\thisx}
        \end{pmatrix}.
    \end{equation}
    Now \eqref{eq:basic-step-rules} ensures the remaining part of the statement.
\end{proof}

Our next step is to simplify $\Test_{i+1}\Precond_{i+1}-\Test_{i+2}\Precond_{i+2}$ in \eqref{eq:convergence-fundamental-condition-iter-h} while keeping the option to accelerate the method when some of the blocks of $H$ exhibit strong monotonicity.
\begin{lemma}
    \label{lemma:local-metric-transfer}
    Fix $i \in \N$, and suppose \eqref{eq:basic-step-rules} holds.
    Then
    $
        \frac{1}{2}\norm{\freevar}_{\Test_{i+1}(\Precond_{i+1}+\GammaLift{i+1})-\Test_{i+2}\Precond_{i+2}}^2
        =
        0
    $ for
    \begin{equation}
    \label{eq:gammalift-def}
        \GammaLift{i+1}\defeq
        \begin{pmatrix}
            2\Tau_i\tilde\Gamma_G & 2\Tau_i \kgradconj{\thisx} \\
            -2\Sigma_{i+1}\kgrad{\nextx} & 2\Sigma_{i+1}\tilde\Gamma_{F^*}
        \end{pmatrix}.
    \end{equation}
\end{lemma}
\begin{proof}
    Using \eqref{eq:basic-step-rules} and \eqref{eq:test-precond-expansion} can write
    \[
        \Test_{i+1}(\Precond_{i+1}+\GammaLift{i+1})-\Test_{i+2}\Precond_{i+2} 
        =
        \begin{pmatrix}
            0 & [\eta_{i+1}\kgrad{\nextx} + \eta_i\kgrad{\thisx}]^* \\
            -\eta_{i+1}\kgrad{\nextx} - \eta_i\kgrad{\thisx} & 0 \\
        \end{pmatrix}.
    \]
    Inserting this into the definition of the weighted norm yields the claim.
\end{proof}

The next somewhat technical lemma estimates the linearizations of $\tilde H_{i+1}$ that are needed to make the abstract algorithm \eqref{eq:ppext} computable for nonlinear $K$.
\begin{lemma}
	\label{lemma:nonlinear-preconditioner-estimate}
	Suppose \cref{ass:k-lipschitz} (locally Lipschitz $\grad K$), \cref{ass:gf} (monotone $\subdiff G$ and $\subdiff F^*$), and \eqref{eq:basic-step-rules} hold. For a fixed $i \in \N$, let $\overnextx \in \neighx_K$ and let $\metricRhoY \ge 0$ be such that $\thisu, \nextu\in \neighu(\metricRhoY)$.
        Also suppose \cref{ass:k-nonlinear} (three-point condition on $K$) holds with $\theta\ge\metricRhoY^{2-p}p^{-p}\omega_i^{-1}\zeta^{1-p}$ for some $\zeta>0$ and $p\in [1,2]$.
	Then
	\begin{multline*}
		\iprod{\tilde H_{i+1}(\nextu)}{\nextu-\realoptu}_{\Test_{i+1}\Precond_{i+1}}
		-\frac{1}{2}\norm{\nextu-\realoptu}_{\Test_{i+1}\GammaLift{i+1}}^2
		\\
		\ge
		\norm{\nexty-\realopty}_{\eta_{i+1}[\Gamma_{F^*}-\tilde\Gamma_{F^*}-(p-1)\zeta\Pnl]}^2
		-\frac{1}{2}\norm{\nextx-\thisx}^2_{\eta_i[\Lambda+L(2+\omega_i)\metricRhoY I]}.
	\end{multline*}
\end{lemma}		
\begin{proof}
	From \eqref{eq:tilde-h}, \eqref{eq:basic-step-rules}, and \eqref{eq:gammalift-def}, we have
	\begin{equation}
	\label{eq:nonlinear-preconditioner-estimate-d-def}
	\begin{aligned}[t]
		D&\defeq \iprod{\tilde H_{i+1}(\nextu)}{\nextu-\realoptu}_{\Test_{i+1}\Precond_{i+1}}
		-\frac{1}{2}\norm{\nextu-\realoptu}_{\Test_{i+1}\GammaLift{i+1}}^2\\
		&=\iprod{H(\nextu)}{\nextu-\realoptu}_{\Test_{i+1}\Step_{i+1}}
		\\
		&\quad	
		+\eta_i\iprod{[\kgrad{\thisx}-\kgrad{\nextx}](\nextx-\realoptx)}{\nexty}
		\\
		&\quad
		+\eta_{i+1}\iprod{K(\nextx)-K(\overnextx)-\kgrad{\thisx}(\nextx-\overnextx)}{\nexty-\realopty}
		\\
		&\quad
		+ \iprod{(\eta_{i+1}\kgrad{\nextx}-\eta_i\kgrad{\thisx})(\nextx-\realoptx)}{\nexty-\realopty}.
		\\
		&\quad 
		-\eta_i\norm{\nextx-\realoptx}_{\tilde\Gamma_G}^2-\eta_{i+1}\norm{\nexty-\realopty}_{\tilde\Gamma_{F^*}}^2.
	\end{aligned}
	\end{equation}
	Since $0 \in H(\realoptu)$, we have $z_G \defeq  - \kgradconj{\realoptx}\realopty \in \subdiff G(\realoptx)$ and $z_{F^*} \defeq  K(\realoptx) \in \subdiff F^*(\realopty)$.
	Using \eqref{eq:basic-step-rules}, we can therefore expand
	\begin{equation*}
        \begin{split}
		\iprod{H(\nextu)}{\nextu-\realoptu}_{\Test_{i+1}\Step_{i+1}}
        & =\eta_i\iprod{\subdiff G(\nextx)-z_G}{\nextx-\realoptx}
		+\eta_{i+1} \iprod{\subdiff F^*(\nexty)-z_{F^*}}{\nexty-\realopty}
        \\
        \MoveEqLeft[-1]
		+\eta_i\iprod{\kgradconj{\nextx} \nexty-\kgradconj{\realoptx}\realopty}{\nextx-\realoptx}
        \\
        \MoveEqLeft[-1]
        +\eta_{i+1}\iprod{K(\realoptx)-K(\nextx)}{\nexty-\realopty}.
        \end{split}
	\end{equation*}
	Using the local (strong) monotonicity of $G$ and $F^*$ (\cref{ass:gf}) and rearranging terms, we obtain 
	\begin{equation}
    	\label{eq:h-strong-monotonicity-estimate}
        \begin{aligned}[t]
		\iprod{H(\nextu)}{\nextu-\realoptu}_{\Test_{i+1}\Step_{i+1}}
        & \ge \eta_i\norm{\nextx-\realoptx}_{\Gamma_G}^2
        +\eta_{i+1}\norm{\nexty-\realopty}_{\Gamma_{F^*}}^2
        \\
        \MoveEqLeft[-1]
		+\eta_i\iprod{\kgrad{\nextx}(\nextx-\realoptx)}{ \nexty}
        \\
        \MoveEqLeft[-1]
        -\eta_i\iprod{\kgrad{\realoptx}(\nextx-\realoptx)}{\realopty}
        \\
        \MoveEqLeft[-1]
        +\eta_{i+1}\iprod{K(\realoptx)-K(\nextx)}{\nexty-\realopty}.
        \end{aligned}
	\end{equation}
	Now we plug the estimate \eqref{eq:h-strong-monotonicity-estimate} into \eqref{eq:nonlinear-preconditioner-estimate-d-def} and rearrange to arrive at 
        \allowdisplaybreaks
	\begin{align*}
        \begin{split}
		D&\ge\eta_i\norm{\nextx-\realoptx}_{\Gamma_G-\tilde\Gamma_G}^2+\eta_{i+1}\norm{\nexty-\realopty}_{\Gamma_{F^*}-\tilde\Gamma_{F^*}}^2
		\\
		&\quad
		-\eta_i\iprod{\kgrad{\realoptx}(\nextx-\realoptx)}{\realopty}+\eta_i\iprod{\kgrad{\thisx}(\nextx-\realoptx)}{\nexty}
		\\
		&\quad
		+\eta_{i+1}\iprod{K(\realoptx)-K(\overnextx)-\kgrad{\thisx}(\nextx-\overnextx)}{\nexty-\realopty}
		\\
		&\quad
		+ \iprod{(\eta_{i+1}\kgrad{\nextx}-\eta_i\kgrad{\thisx})(\nextx-\realoptx)}{\nexty-\realopty}
		\\
		&
		=
		\eta_i\norm{\nextx-\realoptx}_{\Gamma_G-\tilde\Gamma_G}^2+\eta_{i+1}\norm{\nexty-\realopty}_{\Gamma_{F^*}-\tilde\Gamma_{F^*}}^2
		\\
		&\quad
		+\eta_i\iprod{[\kgrad{\thisx}-\kgrad{\realoptx}](\nextx-\realoptx)}{\realopty}
		\\
		&\quad
		+\eta_{i+1}\iprod{K(\realoptx)-K(\nextx)-\kgrad{\nextx}(\realoptx-\nextx)}{\nexty-\realopty}
		\\
		&\quad
		+\eta_{i+1}\iprod{K(\nextx)-K(\overnextx)+\kgrad{\nextx}(\overnextx-\nextx)}{\nexty-\realopty}
		\\
		&\quad
		+\eta_{i+1}\iprod{(\kgrad{\thisx}-\kgrad{\nextx})(\overnextx-\nextx)}{\nexty-\realopty}.
		\end{split}
	\end{align*}
	Applying \cref{ass:k-lipschitz}, \eqref{eq:ass-k-lipschitz-2}, and $\overnextx-\nextx=\omega_i(\nextx-\thisx)$ to the last two terms, we obtain 
	\begin{align*}
		\iprod{K(\nextx)-K(\overnextx)+\kgrad{\nextx}(\overnextx-\nextx)}{\nexty-\realopty}
		&\ge -\frac{L\omega_i^2}{2}\norm{\nextx-\thisx}^2\norm{\nexty-\realopty}_{\Pnl}
		\\
                \shortintertext{and}
		\iprod{(\kgrad{\thisx}-\kgrad{\nextx})(\overnextx-\nextx)}{\nexty-\realopty}
		&\ge -L\omega_i\norm{\nextx-\thisx}^2\norm{\nexty-\realopty}_{\Pnl}.
	\end{align*}
	These estimates together with \eqref{eq:basic-step-rules} and $\nextu\in \neighu(\metricRhoY)$ now imply that
        \begin{equation}
            \label{eq:d-dk-estimate}
            D\ge\eta_iD^K_{i+1}+\eta_{i+1}\norm{\nexty-\realopty}_{\Gamma_{F^*}-\tilde\Gamma_{F^*}}^2
        \end{equation}
	for
	\begin{equation*}
        \begin{split}
		D^K_{i+1}
        & \defeq \iprod{[\kgrad{\thisx}-\kgrad{\realoptx}](\nextx-\realoptx)}{\realopty}
		+\norm{\nextx-\realoptx}_{\Gamma_G-\tilde\Gamma_G}^2-L(1+\omega_i/2)\metricRhoY\norm{\nextx-\thisx}^2
		\\
        \MoveEqLeft[-1]
		-\norm{\nexty-\realopty}_{\Pnl} \norm{K(\realoptx)-K(\nextx)-\kgrad{\nextx}(\realoptx-\nextx)}/\omega_i.
        \end{split}
	\end{equation*}
	Finally, we use \cref{ass:k-nonlinear} to estimate
	\begin{equation}
    	\label{eq:dk-estimate}
        \begin{aligned}[t]
		D^K_{i+1}
        &\ge\theta\norm{K(\realoptx)-K(\nextx)-\kgrad{\nextx}(\realoptx-\nextx)}^p-\frac{1}{2}\norm{\nextx-\thisx}^2_{\Lambda+L(2+\omega_i)\metricRhoY I}
		\\
        \MoveEqLeft[-1]
		-\norm{\nexty-\realopty}_{\Pnl} \norm{K(\realoptx)-K(\nextx)-\kgrad{\nextx}(\realoptx-\nextx)}/\omega_i.
        \end{aligned}
	\end{equation}
	We now use the following Young's inequality for any positive $a,b,p$ and $q$ such that $q^{-1}+p^{-1}=1$:
	\[
		ab=\left(ab^{\frac{2-p}{p}}\right)b^{2\frac{p-1}{p}}
		\le \frac{1}{p}\left(ab^{\frac{2-p}{p}}\right)^p+\frac{1}{q}b^{2\frac{p-1}{p}q}
		=\frac{1}{p}a^pb^{2-p}+\left(1-\frac{1}{p}\right)b^2.
	\]
	With this inequality applied to the last term of \eqref{eq:dk-estimate} for
        \[
            a=(\zeta p)^{-1/2} \norm{K(\realoptx)-K(\nextx)-\kgrad{\nextx}(\realoptx-\nextx)},\qquad
            b=(\zeta p)^{1/2}\norm{\nexty-\realopty}_{\Pnl},
        \]
	and any $\zeta>0$, we arrive at the estimate
	\begin{equation*}
		\begin{aligned}[t]
		D^K_{i+1}
		&\ge\theta\norm{K(\realoptx)-K(\nextx)-\kgrad{\nextx}(\realoptx-\nextx)}^p
		-\frac{1}{2}\norm{\nextx-\thisx}^2_{\Lambda+L(2+\omega_i)\metricRhoY I}
		\\
		\MoveEqLeft[-1]
		-\frac{\norm{\nexty-\realopty}_{\Pnl}^{2-p}}{p^p\omega_i\zeta^{p-1}}\norm{K(\realoptx)-K(\nextx)-\kgrad{\nextx}(\realoptx-\nextx)}^p 
		-\frac{p-1}{\omega_i}\zeta \norm{\nexty-\realopty}_{\Pnl}^2.
		\end{aligned}
	\end{equation*}
	Now observe that $\theta-\norm{\nexty-\realopty}_{\Pnl}^{2-p}(p^p\omega_i\zeta^{p-1})^{-1}\ge\theta-\metricRhoY^{2-p}(p^p\omega_i\zeta^{p-1})^{-1}\ge0$. Therefore
	\begin{equation}
		\label{eq:dk-estimate-p}
		D^K_{i+1} \ge
		-\frac{1}{2}\norm{\nextx-\thisx}^2_{\Lambda+L(2+\omega_i)\metricRhoY I}
		-\frac{p-1}{\omega_i}\zeta \norm{\nexty-\realopty}_{\Pnl}^2.
	\end{equation}	
	Combining this with \eqref{eq:d-dk-estimate} we finally obtain
	\[
		D\ge\norm{\nexty-\realopty}_{\eta_{i+1}[\Gamma_{F^*}-\tilde\Gamma_{F^*}-(p-1)\zeta\Pnl]}^2
		-\frac{1}{2}\norm{\nextx-\thisx}^2_{\eta_i[\Lambda+L(2+\omega_i)\metricRhoY I]},
	\]
    which was our claim.
\end{proof}

We now have all the necessary tools in hand to formulate the main estimate. 
\begin{theorem}
	\label{thm:nonneg-penalty-nlpdhgm}
	Fix $i \in \N$, and suppose \eqref{eq:basic-step-rules} and \cref{ass:k-lipschitz} (locally Lipschitz $\grad K$), \cref{ass:gf} (monotone $\subdiff G$ and $\subdiff F^*$), and \cref{ass:k-nonlinear} (three-point condition on $K$) hold.
	Also suppose $\overnextx \in \neighx_K$ and that $\thisu, \nextu\in \neighu(\metricRhoY)$ for some $\metricRhoY \ge 0$. Furthermore, for $0\le\delta\le\kappa<1$ define
	\[
		S_{i+1} \defeq
		\begin{psmallmatrix}
		\delta\TauTest_i  
		-\eta_i[\Lambda+L(2+\omega_i)\metricRhoY I]
		& 0 \\
		0 & \SigmaTest_{i+1}- \frac{\eta^2_i}{1-\kappa}\kgrad{\thisx}\inv\TauTest_i\kgradconj{\thisx}
		\end{psmallmatrix}.
	\]
	Finally, suppose \cref{ass:k-nonlinear} holds with $\theta\ge\metricRhoY^{2-p}p^{-p}\omega_i^{-1}\zeta^{1-p}$ for some $\zeta>0$.
	Then \eqref{eq:convergence-fundamental-condition-iter-h} is satisfied (for this $i$) if
        \begin{equation}\label{eq:penalty-nlpdhgm}
            \frac{1}{2}\norm{\nextu-\thisu}_{S_{i+1}}^2+\norm{\nexty-\realopty}_{\eta_{i+1}[\Gamma_{F^*}-\tilde\Gamma_{F^*}-(p-1)\zeta\Pnl]}^2 \ge - \Penalty_{i+1}.
        \end{equation}
	In particular, under the above assumptions, we may take $\Penalty_{i+1}=0$ in \eqref{eq:convergence-fundamental-condition-iter-h} provided
	\begin{subequations}
		\label{eq:zero-penalty-rules}
		\begin{align}
		\label{eq:zero-penalty-rules-tautest}
		\TauTest_i & \ge \eta_i\delta^{-1}[\Lambda+L(2+\omega_i)\metricRhoY I],
		\\
		\label{eq:zero-penalty-rules-sigmatest}
		\SigmaTest_{i+1} & \ge \frac{\eta_i^2}{1-\kappa}\kgrad{\thisx}\inv\TauTest_i\kgradconj{\thisx},
		\quad\text{and} 
		\\
		\label{eq:zero-penalty-rules-fstar}
		\Gamma_{F^*} & \ge \tilde\Gamma_{F^*}+(p-1)\zeta\Pnl.
		\end{align}
	\end{subequations}
\end{theorem}
\begin{proof}
    Using the definition of $S_{i+1}$ and \eqref{eq:test-precond-expansion}, we have that
    \[
        \frac{1}{2}\norm{\nextu-\thisu}_{S_{i+1}}^2
        \le\frac{1}{2}\norm{\nextu-\thisu}_{\Test_{i+1} \Precond_{i+1}}^2
        -\frac{1}{2}\norm{\nextx-\thisx}^2_{\eta_i[\Lambda+L(2+\omega_i)\metricRhoY I]}.
    \]
    Since \cref{ass:k-nonlinear} holds with $\theta\ge\metricRhoY^{2-p}p^{-p}\omega_i^{-1}\zeta^{1-p}$ for some $\zeta>0$, we can apply \cref{lemma:nonlinear-preconditioner-estimate} to further bound
    \begin{multline*}
        \frac{1}{2}\norm{\nextu-\thisu}_{S_{i+1}}^2	+\norm{\nexty-\realopty}_{\eta_{i+1}[\Gamma_{F^*}-\tilde\Gamma_{F^*}-(p-1)\zeta\Pnl]}^2
        \\
        \le\frac{1}{2}\norm{\nextu-\thisu}_{\Test_{i+1} \Precond_{i+1}}^2
        +\iprod{\tilde H_{i+1}(\nextu)}{\nextu-\realoptu}_{\Test_{i+1}\Precond_{i+1}}
        -\frac{1}{2}\norm{\nextu-\realoptu}_{\Test_{i+1}\GammaLift{i+1}}^2.
    \end{multline*}
    Using \cref{lemma:local-metric-transfer}, we may insert 
    $\norm{\nextu-\realoptu}_{\Test_{i+1}\GammaLift{i+1}}^2=\norm{\nextu-\realoptu}_{\Test_{i+2}\Precond_{i+2}-\Test_{i+1}\Precond_{i+1}}^2$ and use \eqref{eq:penalty-nlpdhgm} to obtain
    \begin{equation*}
        \begin{split}
            - \Penalty_{i+1}
            &\le\frac{1}{2}\norm{\nextu-\thisu}_{S_{i+1}}^2    +\norm{\nexty-\realopty}_{\eta_{i+1}[\Gamma_{F^*}-\tilde\Gamma_{F^*}-(p-1)\zeta\Pnl]}^2
            \\
            &\le\frac{1}{2}\norm{\nextu-\thisu}_{\Test_{i+1} \Precond_{i+1}}^2
            +\iprod{\tilde H_{i+1}(\nextu)}{\nextu-\realoptu}_{\Test_{i+1}\Precond_{i+1}}
            \\ \MoveEqLeft[-1]
            -\frac{1}{2}\norm{\nextu-\realoptu}_{\Test_{i+2}\Precond_{i+2}-\Test_{i+1}\Precond_{i+1}}^2.
        \end{split}
    \end{equation*}
    Rearranging the terms, we arrive at
    \begin{equation*}
        \iprod{\tilde H_{i+1}(\nextu)}{\nextu-\realoptu}_{\Test_{i+1}\Precond_{i+1}}
        \ge
        \frac{1}{2}\norm{\nextu-\realoptu}_{\Test_{i+2}\Precond_{i+2}-\Test_{i+1}\Precond_{i+1}}^2
        -\frac{1}{2}\norm{\nextu-\thisu}_{\Test_{i+1} \Precond_{i+1}}^2
        -\Penalty_{i+1}.
    \end{equation*}
    Hence, \eqref{eq:convergence-fundamental-condition-iter-h} is satisfied.

    Finally, if in addition the relations \eqref{eq:zero-penalty-rules} are satisfied, then the left-hand side of \eqref{eq:penalty-nlpdhgm} is trivially bounded from below by zero.
\end{proof}

We close this section with some remarks on the conditions \eqref{eq:zero-penalty-rules}:
\begin{itemize}
    \item While \eqref{eq:zero-penalty-rules-tautest} and \eqref{eq:zero-penalty-rules-sigmatest} are stated in terms of $\TauTest_i$ and $\SigmaTest_{i+1}$, they actually provide bounds on the step length operators $\Tau_i$ and $\Sigma_i$: Since $\eta_i I =\TauTest_i\Tau_i=\SigmaTest_i\Sigma_i$ by \eqref{eq:basic-step-rules}, $\TauTest_i$ and $\SigmaTest_{i+1}$ can be eliminated from \eqref{eq:zero-penalty-rules}, and we will do so for scalar step lengths in \cref{sec:scalar}. Thus, while \eqref{eq:basic-step-rules} provides valid update rules for the parameters in \cref{alg:NL-PDHGM}, \eqref{eq:zero-penalty-rules} will provide upper bounds on step lengths under which convergence can be proven.

    \item If $K$ is linear, \eqref{eq:zero-penalty-rules-tautest} reduces to $\TauTest_i \ge 0$ since $\Pnl=0$ and hence $\metricRhoY=0$ and $\Lambda=0$. We can thus take $\kappa=0$, so that \eqref{eq:zero-penalty-rules-sigmatest} turns into an operator analogue of the step length bound $\tau_i\sigma_i \norm{K}^2 < 1$ of \cite{chambolle2010first}. 

    \item Recall from \eqref{eq:neighu-definition} that $\metricRhoY$ only bounds the dual variable on the subspace $\Ynl$. Therefore, most of the requirements for convergence introduced in \cref{sec:scalar-convergence} to account for nonlinear $K$  (e.g., upper bounds on the primal step length, initialization of the dual variable close to a critical point, or the strong convexity of $F^*$ at $\realopty$) will only be required with respect to $\Ynl$. 

    \item Comparing \eqref{eq:zero-penalty-rules} with the requirements of \cite{tuomov-nlpdhgm}, a crucial difference is that in \eqref{eq:zero-penalty-rules-fstar}, $\Gamma_{F^*}$ is allowed to be zero when $p=1$ and hence we do not require strong convexity from $F^*$; see also \cite{tuomov-pdex2nlpdhgm}.
        In fact, for $p=1$ the inequality on $\theta$ in \cref{thm:nonneg-penalty-nlpdhgm} reduces to $\metricRhoY\le\omega_i\theta$. We therefore only need to ensure that the dual variable is initialized close to a critical point within the subspace $\Ynl$.
        If $p\in(1,2]$, the factor $\theta$ imposes a lower bound on the dual factor of strong monotonicity over $\Ynl$. Indeed, the minimal $\zeta$ allowed in \cref{thm:nonneg-penalty-nlpdhgm} is given by
        $\zeta=\metricRhoY^{(2-p)/(p-1)}(p^{p}\omega_i\theta)^{-1/(p-1)}$, and \eqref{eq:zero-penalty-rules-fstar} requires that the factor of strong convexity of $F^*$ at $\realopty$ with respect to the subspace $\Ynl$ is not smaller than this $\zeta$.
    \item Finally, while \eqref{eq:zero-penalty-rules} says nothing about $\Gamma_G$ or $\tilde{\Gamma}_G$, the discussion in and after \cref{ex:0d-dual} indicates that the solution $\realoptx$ to $G(\freevar)+\iprod{K(\freevar)}{\realopty}$ should satisfy a ``nonsmooth'' second-order growth condition to compensate for the nonlinearity of $K$. Therefore, $\tilde{\Gamma}_G$ is implicitly bounded from above by the strong convexity factor of the primal problem in \cref{ass:k-nonlinear}.
\end{itemize}

\subsection{Local step length bounds}
\label{sec:locality}

In order to apply  \cref{lemma:nonlinear-preconditioner-estimate} and therefore \cref{thm:nonneg-penalty-nlpdhgm}, we need one final technical result to ensure that the new iterates $\nextu$ remain in the local neighborhood $\neighu(\metricRhoY)$ of $\realoptu$. The following lemma provides the basis from which we further work in \cref{sec:neighborhood-compatible-iterations} and puts a limit on how far the next iterate can escape from a given neighborhood of $\realoptu$ in terms of bounds on the step lengths. 
\begin{lemma}
	\label{lemma:step-length-bounds-r}
	Fix $i \in \N$.
	Suppose  \cref{ass:k-lipschitz} (locally Lipschitz $\grad K$) holds and $\nextu$ solves \eqref{eq:ppext}.
	For simplicity, assume $\omega_i \le 1$.
	For some $\localRhoX[i], \localRhoY  >0$, and $\delta_{x,i},\delta_y\ge0$, let $\B(\realoptx, \localRhoX[i]+\delta_{x,i}) \subset \neighx_K$,  $\thisx\in\B(\realoptx, \localRhoX[i])$, and  $\thisy\in\B(\realopty, \localRhoY)$.
        If
	\begin{equation}
		\label{eq:step-length-bounds-r}
		\norm{\Tau_i}\le\frac{\delta_{x,i}/2}{\norm{\kgrad{\thisx}}\localRhoY +L\norm{\Pnl \realopty}\localRhoX[i]},
		\quad\text{and}\quad
		\norm{\Sigma_{i+1}}\le\frac{2\delta_{y}(\localRhoX[i]+\delta_{x,i})^{-1}}{L(\localRhoX[i]+\delta_{x,i})+2\norm{\kgrad{\realoptx}}},
	\end{equation}
        then	
	$\nextx,\overnextx\in\B(\realoptx, \localRhoX[i]+\delta_{x,i})$ and $\nexty\in\B(\realopty, \localRhoY+\delta_{y})$.
\end{lemma}
\begin{proof}
	We want to show that the step length conditions \eqref{eq:step-length-bounds-r} imply that
	\[
		\norm{\nextx-\realoptx}\le \localRhoX[i]+\delta_{x,i},
		\quad
		\norm{\overnextx-\realoptx}\le \localRhoX[i]+\delta_{x,i},
		\quad\text{and}\quad
		\norm{\nexty-\realopty}\le \localRhoY+\delta_{y}.
	\]
	We do this by applying the testing argument on the primal and dual variables separately.
	Multiplying \eqref{eq:ppext} by $\Test_{i+1}^*(\nextu-\realoptu)$ with $\TauTest_i=I$ and $\SigmaTest_{i+1}=0$, we get
	\[
		0 \in \iprod{\subdiff G(\nextx)+\kgradconj{\thisx}\thisy}{\nextx-\realoptx}_{\Tau_i}
		+
		\iprod{\nextx-\thisx}{\nextx-\realoptx}.
	\]
	Using the three-point version of Pythagoras' identity,
	\begin{equation}
		\label{eq:standard-identity}
		\iprod{\nextx-\thisx}{\nextx-\realoptx}
		= \frac{1}{2}\norm{\nextx-\thisx}^2
		- \frac{1}{2}\norm{\thisx-\realoptx}^2
		+ \frac{1}{2}\norm{\nextx-\realoptx}^2,
	\end{equation}
	we obtain
	\[
		\norm{\thisx-\realoptx}^2 \in 2\iprod{\subdiff G(\nextx)+\kgradconj{\thisx}\thisy}{\nextx-\realoptx}_{\Tau_i}
		+
		\norm{\nextx-\thisx}^2
		+\norm{\nextx-\realoptx}^2.
	\]
	Using $0 \in \subdiff G(\realoptx) + \kgradconj{\realoptx}\realopty$ and the monotonicity of $\subdiff G$, we then arrive at
	\[
		\norm{\nextx-\thisx}^2 + \norm{\nextx-\realoptx}^2 +
		2\iprod{\kgradconj{\thisx}\thisy-\kgradconj{\realoptx}\realopty}{\nextx-\realoptx}_{\Tau_i}\le
		\norm{\thisx-\realoptx}^2.
	\]
	With $C_x\defeq \norm{\kgradconj{\thisx}\thisy-\kgradconj{\realoptx}\realopty}_{\Tau_i^2}$, this implies that
	\begin{equation}
		\label{eq:step-length-estimate-raw}
		\norm{\nextx-\thisx}^2 + \norm{\nextx-\realoptx}^2\le
		2C_x\norm{\nextx-\realoptx}+\norm{\thisx-\realoptx}^2.
	\end{equation}
	Rearranging the terms and using $\norm{\nextx-\realoptx}\le\norm{\nextx-\thisx}+\norm{\thisx-\realoptx}$ yields
	\[
		(\norm{\nextx-\thisx}-C_x)^2  + \norm{\nextx-\realoptx}^2\le (\norm{\thisx-\realoptx}+C_x)^2,
	\]
        which further leads to
	\begin{equation}
		\label{eq:nextx-realoptx-from-thisx-realoptx}
		\norm{\nextx-\realoptx}\le\norm{\thisx-\realoptx}+C_x.
	\end{equation}
	To estimate the dual variable, we multiply \eqref{eq:ppext} by $\Test_{i+1}^*(\nextu-\realoptu)$ with $\TauTest_{i}=0,\SigmaTest_{i+1}=I$, yielding
	\[
		0 \in \iprod{\subdiff F^*(\nexty)-K(\overnextx)}{\nexty-\realopty}_{\Sigma_{i+1}}+\iprod{\nexty-\thisy}{\nexty-\realopty}.
	\]
	Using $0 \in \subdiff F^*(\realopty) - K(\realoptx)$ and following the steps leading to \eqref{eq:nextx-realoptx-from-thisx-realoptx}, we deduce
	\begin{equation}
		\label{eq:nexty-realoptx-from-thisx-realoptx}
		\norm{\nexty-\realopty} \le \norm{\thisy-\realopty}+C_y
	\end{equation}
	with $C_y\defeq \norm{K(\realoptx)-K(\overnextx)}_{\Sigma_{i+1}^2}$.
	
	We now proceed to deriving bounds on $C_x$ and $C_y$ with the goal of bounding \eqref{eq:nextx-realoptx-from-thisx-realoptx} and \eqref{eq:nexty-realoptx-from-thisx-realoptx} from above.
	Using \cref{ass:k-lipschitz} and arguing as in \eqref{eq:ass-k-lipschitz-2}, we estimate
	\begin{align}
		\label{eq:cx-estimate}
		C_x&\le\norm{\Tau_i}(\norm{\kgrad{\thisx}}\norm{\thisy-\realopty}+L\norm{\Pnl \realopty}\norm{\thisx-\realoptx}) =: R_x,\\
                \intertext{and, if $\overnextx\in\neighx_K$,}
		\label{eq:cy-estimate}
		C_y&\le\norm{\Sigma_{i+1}}(L\norm{\overnextx-\realoptx}/2+\norm{\kgrad{\realoptx}})\norm{\overnextx-\realoptx} =: R_y.
	\end{align}
        We thus need to verify first that $\overnextx\in\neighx_K$. By definition,
	\begin{equation*}
	\begin{split}
		\norm{\overnextx-\realoptx}^2&=\norm{\nextx-\realoptx+\omega_i(\nextx-\thisx)}^2\\
		&=\norm{\nextx-\realoptx}^2+\omega_i^2\norm{\nextx-\thisx}^2+2\omega_i\iprod{\nextx-\realoptx}{\nextx-\thisx}\\
		&=(1+\omega_i)\norm{\nextx-\realoptx}^2+\omega_i(1+\omega_i)\norm{\nextx-\thisx}^2-\omega_i\norm{\thisx-\realoptx}^2\\
		&\le(1+\omega_i)(\norm{\nextx-\realoptx}^2+\norm{\nextx-\thisx}^2)-\omega_i\norm{\thisx-\realoptx}^2.
	\end{split}
	\end{equation*}
    Now, the bound \eqref{eq:step-length-bounds-r} on $\Tau_i$ together with the definition of $R_x$ implies that $C_x \leq R_x\leq \delta_{x,i}/2$.
	Applying \eqref{eq:step-length-estimate-raw} and \eqref{eq:nextx-realoptx-from-thisx-realoptx}, we obtain from this that
	\[
		\norm{\overnextx-\realoptx}^2\le 4C_x\norm{\nextx-\realoptx}+\norm{\thisx-\realoptx}^2
                \le 4C_x(\localRhoX[i]+C_x)+\localRhoX[i]^2 \leq (\localRhoX[i]+\delta_{x,i})^2.
	\]
	In addition, \eqref{eq:nextx-realoptx-from-thisx-realoptx} shows that $\norm{\nextx-\realoptx}\le \localRhoX[i]+\delta_{x,i}$ as well.
    Similarly, the bound \eqref{eq:step-length-bounds-r} on $\Sigma_{i+1}$ implies that $C_y\leq R_y\leq \delta_y$, and hence \eqref{eq:nexty-realoptx-from-thisx-realoptx} and \eqref{eq:cy-estimate} lead to $\norm{\nexty-\realopty}\le \localRhoY+\delta_{y}$, completing the proof.
\end{proof}
Note that only the radii $\localRhoX[i]$, $\localRhoX[i]+\delta_{x,i}$ of the primal neighborhoods depend on the iteration, and we will later seek to control these based on actual convergence estimates.

\begin{remark}
	\label{rmk:no-locality}
	Suppose that $\neighx_K=\neighx_G=X$. Then we can take $\delta_{x,i}$ arbitrarily large in order to satisfy the bound on $\Tau_i$ while still satisfying $\B(\realoptx, \localRhoX[i]+\delta_{x,i}) \subset \neighx_K$.
    On the other hand, the bound on $\Sigma_{i+1}$ will go to zero as $\delta_{x,i}\to \infty$, which seems at first very limiting. However, we observe from the proof that this bound is actually not required to satisfy $\nextx,\overnextx\in\B(\realoptx, \localRhoX[i]+\delta_{x,i})$. Furthermore, if $\Dom F^*$ is bounded and we take $\localRhoY$ large enough that $\Dom F^*\subseteq \B(\realopty, \localRhoY)$, the property $\nexty\in\B(\realopty, \localRhoY+\delta_{y})$ is automatically satisfied by the iteration and does not require dual step length bounds.
    Hence if $\neighx_K=\neighx_G=X$ and $\Dom F^*$ is bounded, we can expect global convergence.  We will return to this topic in \cref{rmk:global-convergence}.
\end{remark}

\section{Refinement to scalar step lengths}
\label{sec:scalar}
To derive convergence rates, we now simplify \cref{thm:nonneg-penalty-nlpdhgm} to scalar step lengths. Specifically, we assume for some scalars $\gamma_G,\gamma_{\LIN},\gamma_{\NL}, \tau_i, \tauTest_i, \sigma_i, \sigmaTest_i \ge 0$, and $\lambda \in \R$ the structure
\begin{equation}
\label{eq:scalar-step-lengths}
\left\{
\begin{aligned}
    \Tau_i & = \tau_i I,
    &
    \TauTest_i & = \tauTest_i I,
    &
    \Gamma_G & = \gamma_G I,
    &
    &&&
	\\
	\Sigma_i & = \sigma_{i} I,
	&
	\SigmaTest_i & = \sigmaTest_{i} I,
	&
	\Gamma_{F^*} & =  \gamma_{\LIN}P_\LIN+\gamma_{\NL}\Pnl,
	&
	\text{and}&
	&
    \Lambda & = \lambda I.
\end{aligned}
\right.
\end{equation}
Consequently, the preconditioning, step length, and testing operators simplify to
\begin{align*}
    \Precond_{i+1} & \defeq
    \begin{pmatrix}
        I & -\tau_i \kgradconj{\thisx} \\
        -\omega_i \sigma_{i+1} \kgrad{\thisx} & I
    \end{pmatrix},\\
    \Step_{i+1} &\defeq
    \begin{pmatrix}
        \tau_i I & 0 \\
        0 & \sigma_{i+1} I
    \end{pmatrix}
    \quad\text{and}\quad
    \Test_{i+1} \defeq
    \begin{pmatrix}
        \tauTest_i I & 0 \\
        0 & \sigmaTest_{i+1} I
    \end{pmatrix}.
\end{align*}
This reduces \eqref{eq:ppext} to \cref{alg:NL-PDHGM}, which for convex, proper, lower semicontinuous $G$ and $F^*$ is always solvable for the iterates $\{\thisu \defeq (\thisx, \thisy)\}_{i \in \N}$. Before proceeding to the main results, we state next all our assumptions in scalar form. We then derive in \cref{sec:scalar-convergence} our main convergence rates results for \cref{alg:NL-PDHGM} under specific update rules depending on monotonicity properties of $G$ and $F^*$. The final \cref{sec:neighborhood-compatible-iterations} is devoted to giving sufficient conditions for the scalar version of the assumptions of \cref{lemma:step-length-bounds-r} to hold.

\subsection{General derivations and assumptions}
\label{sec:scalar-general}

Under the setup \eqref{eq:scalar-step-lengths}, the update rules \eqref{eq:basic-step-rules} and the conditions \eqref{eq:zero-penalty-rules} simplify to
\begin{subequations}
\label{eq:scalar-rules}
\allowdisplaybreaks
\begin{align}
    \label{eq:scalar-step-rules0}
    \omega_i & = \eta_i/\eta_{i+1}, 
    &
    \eta_i & = \sigmaTest_{i}\sigma_{i} = \tauTest_i\tau_i,
    \\
    \label{eq:scalar-test-update}
    \tauTest_{i+1} & =\tauTest_i(1+2\tau_i\tilde\gamma_G),
    &
    \sigmaTest_{i+2} & =\sigmaTest_{i+1}(1+2\sigma_{i+1}\tilde\gamma_{F^*}),
    \\
    \label{eq:scalar-step-rules-test}
    \tauTest_i & \ge \eta_i\delta^{-1}(\lambda+(\omega_i+2)L\metricRhoY),
    &
	\sigmaTest_{i+1} & \ge \frac{\eta_i^2 \inv \tauTest_i}{1-\kappa} \norm{\grad K(\thisx)}^2,
    \\
    \label{eq:scalar-gamma-rules}
    \gamma_{\LIN} & \ge \tilde\gamma_{F^*}, 
    &
    \gamma_{\NL}& \ge \tilde\gamma_{F^*}+(p-1)\zeta,
\end{align}
\end{subequations}
for some $\eta_i>0$, $p\in[1,2]$, $\zeta>0$, $0\leq \delta\leq \kappa<1$, and $\tilde \gamma_{F^*}$, for which we will from now on further assume $\tilde\gamma_{F^*}\geq 0$. To formulate the scalar version of \cref{ass:k-nonlinear}, we also introduce the corresponding factor $\tilde \gamma_{G}\geq 0$; see \cref{ass:main-for-convergence}\,\ref{item:main-for-convergence:k-nonlinear} below.

Let us comment on these relations in turn.  The conditions \eqref{eq:scalar-step-rules0} and \eqref{eq:scalar-test-update} limit the rate of growth of the testing parameters -- and thus the convergence rate -- and set basic coupling conditions for the step length parameters. They are virtually unchanged from the standard case of linear $K$; see \cite[Example 3.2]{tuomov-proxtest}.

The conditions in \eqref{eq:scalar-step-rules-test} are essentially step length bounds: 
Substituting $\eta_i=\tauTest_i\tau_i$ and  $\eta_i^2=\tauTest_i\tau_i \sigmaTest_{i}\sigma_{i}$ in \eqref{eq:scalar-step-rules-test}, we obtain
\begin{equation}
	\label{eq:scalar-tau-sigma-estimate-rhoy}
	\tau_i\le \frac{\delta}{\lambda+(\omega_i+2)L\metricRhoY},
	\quad\text{and}\quad
	\sigma_{i}\tau_i\le\frac{1-\kappa}{R_K^2},
\end{equation}
where $R_K=\sup_{\neighx}\norm{\grad K(x)}$. In the latter bound, we also used $\sigmaTest_{i+1}\ge \sigmaTest_{i}$, which follows from \eqref{eq:scalar-test-update} and $\tilde\gamma_{F^*} \ge 0$. We point out that this  condition is simply a variant for nonlinear $K$ of the standard initialization condition $\tau\sigma \norm{K}^2 < 1$ for the PDHGM for linear $K$.
From \eqref{eq:scalar-tau-sigma-estimate-rhoy}, we see that we need to initialize and keep the dual iterates at a known finite distance $\metricRhoY$ from $\realopty$. It also individually bounds the primal step length based on further properties of the specific saddle-point problem. The nonconvexity enters here via the factor $\lambda$ from the three-point condition on $K$ (\cref{ass:k-nonlinear}).

The final condition \eqref{eq:scalar-gamma-rules} bounds the acceleration parameters $\tilde\gamma_{F^*}$ based on the actually available strong monotonicity minus any penalties we get from nonlinearity of $K$ if \cref{ass:k-nonlinear} is satisfied with $p\in(1,2]$. However, $K$ may contribute to strong monotonicity of the primal problem \emph{at} $\realoptx$, so it can in some specific problems be possible to choose $\tilde \gamma_G > \gamma_G$.

Before we will further refine these bounds in the following sections, we collect the scalar refinements of all the structural assumptions of \cref{sec:nlpdhgm-analysis}.
\begin{assumption}
	\label{ass:main-for-convergence}
	Suppose $G: X \to \extR$ and $F^* \to \extR$ are convex, proper, and lower semicontinuous, and  $K \in C^1(X; Y)$.
        Furthermore:
        \begin{enumerate}[label=(\roman*)]
            \item \label{item:main-for-convergence-lipschitz-gradk} (\emph{locally Lipschitz $\grad K$})
                There exists $L \ge 0$ with $\|\kgrad{x}-\kgrad{x'}\|\le L\|x-x'\|$ for any $x,x'\in\neighx_K$.
            \item  \label{item:main-for-convergence-bounded-gradk}
                (\emph{locally bounded $\grad K$})
                There exists $R_K>0$ with $\sup_{x\in\neighx_K}\norm{\grad K(x)}\le R_K$.
            \item \label{item:main-for-convergence:monotone-g-f} (\emph{monotone $\subdiff G$ and $\subdiff F^*$})
                The mapping $\subdiff G$ is ($\gamma_GI$-strongly) monotone at $\realoptx$ for $-\kgradconj{\realoptx}\realopty$ in $\neighx_G$ with $\gamma_G \ge 0$; and the mapping $\subdiff F^*$ is ($\gamma_{\LIN}P_\LIN+\gamma_{\NL}\Pnl$-strongly) monotone at $\realopty$ for $\realopt{\xi}+K(\realoptx)$ in $\neighy_{F^*}$ with $\gamma_{\LIN},\gamma_{\NL} \ge 0$.
            \item \label{item:main-for-convergence:k-nonlinear}
                (\emph{three-point condition on $K$}) 
                For some $p\in[1,2]$, some $\lambda,\theta \ge0$, and any $x,x'\in\neighx_K$,
                \begin{multline*}
                    \iprod{[\kgrad{x'}-\kgrad{\realoptx}]^*\realopty}{x-\realoptx}+(\gamma_G-\tilde\gamma_G)\norm{x-\realoptx}^2
                    \\
                    \ge \theta\norm{K(\realoptx)-K(x)-\kgrad{x}(\realoptx-x)}^p-\frac{\lambda}{2}\norm{x-x'}^2.
                \end{multline*}
            \item \label{item:main-for-convergence:neighbourhood}
                (\emph{neighborhood-compatible iterations})
                The iterates of \cref{alg:NL-PDHGM} satisfy 
                $\{\thisu\}_{i \in \N}\in \neighu (\metricRhoY)$ and $\{\overnextx\}_{i \in \N}\in \neighx_K$ for some $\metricRhoY \ge 0$, where $\neighu(\metricRhoY)$ is given by \eqref{eq:neighu-definition}.
        \end{enumerate}
\end{assumption}

We again close with remarks on the assumptions.
\begin{itemize}
    \item Assumptions \ref{item:main-for-convergence-lipschitz-gradk} and \ref{item:main-for-convergence-bounded-gradk} are standard assumptions in nonlinear optimization of smooth functions.
    \item Assumption \ref{item:main-for-convergence:monotone-g-f} is always satisfied due to the assumed convexity of $G$ and $F^*$; it only becomes restrictive under the additional requirement that $\gamma_G$ or $\gamma_L,\gamma_{NL}$ are positive, which will be needed to derive convergence rates in the next section. However, we stress that  we never require the functions to be strongly monotone globally; rather we only require the strong monotonicity \emph{at $\realoptx$} or \emph{at $\realopty$}. Furthermore -- and this is a crucial feature of our operator-based analysis --  we can split the strong monotonicity of $\subdiff F^*$ on the subspaces $\Ynl$ and $\Ylin$. We will more depend on the former, which is often automatic as in the example given in the introduction.
    \item We have already elaborated on \ref{item:main-for-convergence:k-nonlinear} in \cref{ex:0d-dual} and \cref{pr:strongly-convex-primal}. In particular, this condition holds if the saddle-point problem satisfies a standard second-order growth condition at $\realoptu$ with respect to the primal variable.
    \item Finally, only assumption \ref{item:main-for-convergence:neighbourhood} is specific to the actual algorithm (i.e., the choice of step sizes); it requires that the iterates of \cref{alg:NL-PDHGM} remain in the neighborhood in which the first four assumptions are valid.  We will prove in \cref{sec:neighborhood-compatible-iterations} that this can be guaranteed under additional bounds on the step lengths. Moreover, we will demonstrate in \cref{rmk:global-convergence} that \cref{ass:main-for-convergence}\,\ref{item:main-for-convergence:neighbourhood} is always satisfied for a specific class of problems, for which we therefore obtain global convergence.
\end{itemize}

\subsection{Convergence results}
\label{sec:scalar-convergence}

We now come to the core of our work, where we apply the analysis of the preceding sections to derive convergence results under explicit step lengths rules. We start with a weak convergence result that requires no (partial) strong monotonicity, which however needs to be replaced with additional assumptions on $K$. 
\begin{theorem}[weak convergence]
	\label{thm:weak-convergence-nlpdhgm}
	Suppose \cref{ass:main-for-convergence} holds for some $L\geq 0$, $R_K>0$, $\lambda\geq 0$, and $\metricRhoY \geq 0$ and choose the step lengths as $\tau_i \equiv \tau$, $\sigma_i \equiv \sigma$, and $\omega_i \equiv 1$.
        Assume that for some $\zeta>0$ and $p\in[1,2]$ that 
        \begin{equation}\label{eq:weak-convergence-p}
            \gamma_{\NL}\ge (p-1)\zeta\quad \text{and} \quad   \theta\ge\metricRhoY^{2-p}p^{-p}\zeta^{1-p},
        \end{equation}
        and for some $0<\delta<\kappa<1$ that
	\begin{equation}
	\label{eq:step-weak-convergence}
		0<\tau< \frac{\delta}{3L\metricRhoY+\lambda}
		\quad\text{and}\quad
		0<\sigma\tau<\frac{1-\kappa}{R_K^2}.
	\end{equation}
    Furthermore, suppose that
    \begin{enumerate}[label=(\roman*)]
        \item\label{item:weak-convergence-cont} 
            $\thisx \weakto \bar{x}$ implies that $\grad K(\thisx)x\rightarrow\grad K(\bar{x})x$ for all $x\in X$,
    \end{enumerate}
    and either
    \begin{enumerate}[label=(ii\alph*)]
        \item\label{item:weak-convergence-maxmon}
            $H(u)$ is maximally monotone in $\neighu (\metricRhoY)$;
        \item\label{item:weak-convergence-weakstrong}
            the mapping $(x,y)\mapsto(\kgradconj{x}y,K(x))$ is weak-to-strong continuous in $\neighu(\metricRhoY)$; or
        \item\label{item:weak-convergence-weakweak} 
            the mapping $(x,y)\mapsto(\kgradconj{x}y,K(x))$ is weak-to-weak continuous, and in addition \cref{ass:main-for-convergence}\,\ref{item:main-for-convergence:monotone-g-f} (monotone $\subdiff G$ and $\subdiff F^*$) and \cref{ass:main-for-convergence}\,\ref{item:main-for-convergence:k-nonlinear} (three-point condition on $K$) hold at any weak limit $\bar{u}=(\bar{x},\bar{y})$ of $\{\thisu\}$ in addition to $\realoptu$ with $\theta\ge(2\metricRhoY)^{2-p}p^{-p}\zeta^{1-p}$.
    \end{enumerate}
    Then the sequence $\{\thisu\}$ generated by \eqref{eq:ppext} converges weakly to some $\bar{u} \in \inv H(0)$ (possibly different from $\realoptu$).
\end{theorem}
\begin{proof}
    We wish to apply \cref{prop:rateless-varying} and therefore need to verify its assumptions. For the basic assumptions of \eqref{eq:convergence-fundamental-condition-iter-h} and the self-adjointness of $\Test_{i+1}\Precond_{i+1}$, we will use \cref{thm:nonneg-penalty-nlpdhgm} together with \cref{lemma:zimi-estim}.
    Most of their assumptions are directly verified by \cref{ass:main-for-convergence}; it only remains to verify \eqref{eq:basic-step-rules} and \eqref{eq:zero-penalty-rules}, which reduce to \eqref{eq:scalar-rules} in the scalar case. By taking $\tilde\gamma_G=\tilde\gamma_{F^*}=0$ and any positive constants $\sigmaTest$ and $\tauTest$ such that $\sigmaTest\sigma=\tauTest\tau$, the relations \eqref{eq:scalar-step-rules0}, \eqref{eq:scalar-test-update}, and \eqref{eq:scalar-gamma-rules} hold. Furthermore, since $\omega_i \equiv 1$, \eqref{eq:step-weak-convergence} is equivalent to \eqref{eq:scalar-tau-sigma-estimate-rhoy}. This yields \eqref{eq:scalar-step-rules-test}, completing the verification of \eqref{eq:scalar-rules} and thus the conditions of \cref{thm:convergence-result-main-h,prop:rateless-varying}. 

    Since the inequalities in \eqref{eq:step-weak-convergence} are strict, we can deduce from \cref{thm:nonneg-penalty-nlpdhgm} that \eqref{eq:penalty-nlpdhgm} even holds for $\Penalty_{i+1}\le-\hat{\delta}\norm{\nextu-\thisu}^2$ for some $\hat{\delta}>0$. Combining \eqref{eq:test-precond-expansion-estimate} and \eqref{eq:step-weak-convergence}, we thus obtain that condition \ref{item:weak-zm-opers} of \cref{prop:rateless-varying} holds. Furthermore, the condition \ref{item:weak-zm-a-limit} follows from the assumed constant step lengths and the assumption \ref{item:weak-convergence-cont}.

    It remains to show the condition \ref{item:weak-zm-h-limit} of \cref{prop:rateless-varying}. First, if the assumption \ref{item:weak-convergence-maxmon} holds, the inclusion in condition \ref{item:weak-zm-h-limit} follows directly from the fact that maximally monotone operators have sequentially weakly–strongly closed graphs \cite[Proposition 20.38]{bauschke2017convex}. 

    The two other cases are more difficult to verify. First, we note that for any $\nextx\weakto \bar{x}$ and $\nexty\weakto \bar{y}$ we have $\Step_{i+1}\equiv \Step$, and \eqref{eq:ppext} implies that $v_{i+1}\in \Step A(\nextu)$ for 
    \begin{align}
        \nonumber
        A(\nextu)&\defeq  
        \begin{pmatrix}
            \subdiff G(\nextx)-\gamma_{G}(\nextx-\bar{x})\\
            \subdiff F^*(\nexty) -\gamma_{\NL}\Pnl(\nexty-\bar{y})
        \end{pmatrix}
        \\
        \intertext{and}
        \label{eq:weak-converging-subdiff}
        v_{i+1} &\defeq
        \begin{multlined}[t]
            \Step
            \begin{pmatrix}
                - \kgradconj{\nextx}\nexty-\gamma_{G}(\nextx-\bar{x})\\
                K(\nextx)-\gamma_{\NL}\Pnl(\nexty-\bar{y})
            \end{pmatrix}
            \\
            -
            \Step
            \begin{pmatrix}
                [\kgrad{\thisx}-\kgrad{\nextx}]^*{\nexty} \\
                K(\nextx)-K(\overnextx)-\kgrad{\thisx}(\nextx-\overnextx).
            \end{pmatrix}
            -\Precond_{i+1}(\nextu-\thisu).
        \end{multlined}
    \end{align}
    Therefore, we need to show that 
    \[
        v_{i+1}\weakto \bar{v}\defeq
        \Step\begin{pmatrix}
            - \kgradconj{\bar{x}}\bar{x}\\
            K(\bar{x})
        \end{pmatrix}
        \quad\text{and}\quad
        \bar{v}\in A(\bar{u}),
    \]
    which by construction is equivalent to the inclusion $\bar{u} \in \inv H(0)$. Note that due to \cref{ass:main-for-convergence}\,\ref{item:main-for-convergence:monotone-g-f}, $A$ is maximally monotone since it only involves subgradient mappings of proper, convex, and lower semicontinuous functions.
    From \cref{thm:convergence-result-main-h}, we obtain \eqref{eq:convergence-result-main-h}, which in turn implies that $\Test_{i+1}\Precond_{i+1}(\nextu-\thisu) \to 0$. The scalar case of \cref{lemma:zimi-estim} together with the condition $0<\delta<\kappa<1$ and positive $\sigmaTest$ and $\tauTest$ then results in $\norm{\nextu-\thisu} \to 0$, so the last two terms in \eqref{eq:weak-converging-subdiff} go to zero.
    We therefore only have to consider the first term, for which we make a case distinction:

    If assumption \ref{item:weak-convergence-weakstrong} holds, we obtain that $v_{i+1}\to \bar{v}$, and the required inclusion $\bar{v}\in A(\bar{u})$ follows from the fact that the graph of the maximally monotone operator $A$ is sequentially weakly–strongly closed; see \cite[Proposition 16.36]{bauschke2017convex}. 

    If assumption \ref{item:weak-convergence-weakweak} holds, then only $v_{i+1}\weakto \bar{v}$. In this case, we can apply the Brezis--Crandall--Pazy Lemma \cite[Corollary 20.59 (iii)]{bauschke2017convex} to obtain the required inclusion under the additional condition that $\limsup_{i\to \infty}~\iprod{u_i-\bar{u}}{v_i-\bar{v}}\le 0$.
    In our case,
    $\limsup_{i\rightarrow\infty}~\iprod{u_i-\bar{u}}{v_i-\bar{v}} = \limsup_{i\rightarrow\infty}~q_i$
    for
    \begin{multline*}
        q_i \defeq
        \iprod{\kgradconj{\bar{x}}\bar{y}-\kgradconj{\nextx}\nexty}{\nextx-\bar{x}}
        +\iprod{K(\nextx)-K(\bar{x})}{\nexty-\bar{y}}
        \\
        -\gamma_{\NL}\norm{\nexty-\bar{y}}_{\Pnl}^2-\gamma_{G}\norm{\nextx-\bar{x}}^2.
    \end{multline*}
    Note that $\norm{\nexty-\bar{y}}_{\Pnl} \le 2\metricRhoY$ because $\norm{\nexty-\realopty}_{\Pnl}, \norm{\realopty-\bar{y}}_{\Pnl} \le \metricRhoY$.
    With this, \eqref{eq:ass-k-lipschitz-2}, and both \cref{ass:main-for-convergence}\,\ref{item:main-for-convergence:monotone-g-f} and \ref{item:main-for-convergence:k-nonlinear} at $\bar{u}$, we similarly to \eqref{eq:dk-estimate-p} estimate
    \begin{equation}
        \label{eq:weak-convergence-q}
        \begin{aligned}[t]
            q_i
            &=\iprod{K(\nextx)-K(\bar{x})+\kgrad{\nextx}(\bar{x}-\nextx)}{\nexty-\bar{y}}
            \\
            &\quad
            -(\iprod{(\kgrad{\thisx}-\kgrad{\bar{x}})(\nextx-\bar{x})}{\bar{y}}+\gamma_{G}\norm{\nextx-\bar{x}}^2)
            \\
            &\quad
            -\gamma_{\NL}\norm{\nexty-\bar{y}}_{\Pnl}^2+\iprod{(\kgrad{\thisx}-\kgrad{\nextx})(\nextx-\bar{x})}{\bar{y}}
            \\
            &\le
            (\norm{\nexty-\bar{y}}_{\Pnl}^{2-p}p^{-p}\zeta^{1-p}-\theta) \norm{K(\bar{x})-K(\nextx)-\kgrad{\nextx}(\bar{x}-\nextx)}^p
            \\
            &\quad
            +((p-1)\zeta-\gamma_{\NL})\norm{\nexty-\bar{y}}_{\Pnl}^2+O(\norm{\nextx-\thisx}).
        \end{aligned}
    \end{equation}
    Since 
    \[
        \norm{\nexty-\bar{y}}_{\Pnl}^{2-p}p^{-p}\zeta^{1-p}-\theta\le(2\metricRhoY)^{2-p}p^{-p}\zeta^{1-p}-\theta\le0,
    \]
    $(p-1)\zeta-\gamma_{\NL}\le0$, and $\Test_{i+1}\Precond_{i+1}(\nextu-\thisu) \to 0$, we obtain that $\limsup_{i\rightarrow\infty}~q_i \le 0$. The Brezis--Crandall--Pazy Lemma thus yields the desired inclusion $\bar{v}\in A(\bar{u})$.

    Hence in all three cases, the condition \ref{item:weak-zm-h-limit} of \cref{prop:rateless-varying} holds with $\thisu \weakto \bar{u} \in \inv H(0)$, which completes the proof.
\end{proof}

\begin{remark}
    It is instructive to consider the two limiting cases $p=1$ and $p=2$ in the conditions \eqref{eq:weak-convergence-p} of \cref{thm:weak-convergence-nlpdhgm}:
    \begin{enumerate}[label=(\roman*)]
        \item
            If $p=1$, the condition on $\gamma_{\NL}$ is trivially satisfied, while the one on $\theta$ reduces to $\metricRhoY\le\theta$.
            Therefore, we only require the dual variable to be initialized close to $\realopty$ (and only when projected into the subspace $\Ynl$).

        \item
            In contrast, if $p=2$, the condition on $\theta$ does not involve $\metricRhoY$ and hence there will be no dual initialization bound; on the other hand, $\zeta \ge (4\theta)^{-1}$ will be required. Therefore, we need $F^*$ to be strongly convex with the factor $\gamma_{\NL}\ge(4\theta)^{-1}$, but only \emph{at $\realopty$} within the subspace $\Ynl$.
    \end{enumerate}

    The remaining cases $p\in (1,2)$ can be seen as an interpolation between these conditions.
	The same observations hold for the other results of this section.  
\end{remark}

We now turn to convergence rates under strong monotonicity assumptions, starting with the case that merely $G$ is strongly convex. Since we obtain   \emph{a fortiori} strong convergence from the rates, we do not require the additional assumptions on $K$ needed to apply \cref{prop:rateless-varying}; on the other hand, we only obtain convergence of the primal iterates. In the proof of the following result, observe how the step length choice follows directly from having to satisfy \eqref{eq:scalar-test-update} and the desire to keep $\sigma_i\tau_i$ constant to satisfy \eqref{eq:scalar-step-rules-test} via the bound \eqref{eq:scalar-tau-sigma-estimate-rhoy} on the initial choice.
\begin{theorem}[convergence rates under acceleration]
    \label{thm:acceleration-nlpdhgm}
    Suppose \cref{ass:main-for-convergence} holds for some $L\geq0$, $R_K>0$, $\lambda\geq 0$, and $\metricRhoY\geq0$ with $\tilde{\gamma}_G>0$, $\gamma_{\NL}\ge (p-1)\zeta$, and $\theta\ge\metricRhoY^{2-p}p^{-p}\sqrt{1+2\tau_0\tilde\gamma_G}\zeta^{1-p}$ for some $\zeta>0$ and $p\in[1,2]$.
    Choose
    \begin{equation}
        \label{eq:acceleration-updates}
        \tau_{i+1}=\tau_{i}\omega_{i},
        \quad
        \sigma_{i+1}=\sigma_{i}/\omega_{i},
        \quad
        \omega_i=1/\sqrt{1+2\tau_i\tilde{\gamma}_G}
    \end{equation}
    with 
    \begin{equation}
        \label{eq:acceleration-init}
        0<\tau_0\le \frac{\delta}{3L\metricRhoY+\lambda},
        \quad\text{and}\quad
        0<\tau_0\sigma_0\le\frac{1-\kappa}{R_K^2}.
    \end{equation} 
    for some $0<\delta\le\kappa<1$.
    Then $\norm{\thisx-\realoptx}^2$ converges to zero at the rate $O(1/N^2)$.
\end{theorem}
\begin{proof}
    The first stage of the proof is similar to \cref{thm:weak-convergence-nlpdhgm}, where we verify \eqref{eq:scalar-rules} to use \cref{thm:nonneg-penalty-nlpdhgm} (but need not apply \cref{prop:rateless-varying}).
    Since we do not assume any (partial) strong convexity of $F^*$, we have to take $\tilde\gamma_{F^*}=0$, and thus \eqref{eq:scalar-gamma-rules} is satisfied by assumption. 
    Note that by \eqref{eq:acceleration-updates}, we have $\sigma_{i}\tau_i=\sigma_0\tau_0$ for all $i\in\N$, $\omega_i<1$, and $\tau_{i+1}<\tau_0$.
    Then \eqref{eq:acceleration-init} leads to \eqref{eq:scalar-tau-sigma-estimate-rhoy}, which is equivalent to  \eqref{eq:scalar-step-rules-test}.
    Taking now $\eta_i\defeq \sigma_i>0$, $\sigmaTest_i \equiv 1$, and $\tauTest_i \defeq \sigma_0\tau_0\tau_i^{-2}>0$, \eqref{eq:scalar-step-rules0} and \eqref{eq:scalar-test-update} follow from \eqref{eq:acceleration-updates} since
    $\eta_{i}\defeq \sigma_{i}=\sigma_{i+1}\omega_{i}=\eta_{i+1}\omega_{i}$
    and
    \[
        \tauTest_i\tau_i=\frac{\sigma_0\tau_0}{\tau_i}=\frac{\sigma_i\tau_i}{\tau_i}=\sigmaTest_i\sigma_i
        \quad\text{and}\quad
	    \tauTest_{i+1} \defeq \frac{\sigmaTest\sigma_0\tau_0}{\tau_{i+1}^2}
	    =\frac{\sigmaTest\sigma_0\tau_0}{\tau_{i}^2\omega_i^2}=\tauTest_{i}(1+2\tau_i\tilde{\gamma}_G).	    
    \]
    Furthermore, \eqref{eq:acceleration-updates} also implies that
    \[
        1/\omega_i\le1/\omega_0=\sqrt{1+2\tau_0\tilde{\gamma}_{G}},
    \]
    and together with our assumption on $\theta$ we obtain that $\theta\ge\metricRhoY^{2-p}p^{-p}\omega_i^{-1}\zeta^{1-p}$. We can thus apply \cref{thm:nonneg-penalty-nlpdhgm,thm:convergence-result-main-h} to arrive at \eqref{eq:convergence-result-main-h} with each $\Penalty_{i+1}\le0$. 

    We now estimate the convergence rate from \eqref{eq:convergence-result-main-h} by bounding $\Test_{N+1}\Precond_{N+1}$ from below. Using \cref{lemma:zimi-estim}, we obtain that
    \begin{equation}
        \label{eq:acceleration-result}
        \delta \tauTest_N \norm{x^N-\realoptx}^2 \le\norm{u^0-\realoptu}^2_{\Test_{1}\Precond_{1}}.
    \end{equation}
    But from \cite[Corollary 1]{chambolle2010first}, we know that $\tau_N =\mathcal{O}(N^{-1})$ as $N\to\infty$ and hence $\tauTest_N \sim \tau_N^{-2} = \mathcal{O}(N^{2})$ by our choice of $\tauTest_N$, which yields the desired convergence rate.
\end{proof}

If both $\subdiff G$ and $\subdiff F^*$ are strongly monotone, \cref{alg:NL-PDHGM} with constant step lengths leads to convergence of both primal and dual iterates at a linear rate.
\begin{theorem}[linear convergence]
    \label{thm:linear-convergence-nlpdhgm}
    Suppose \cref{ass:main-for-convergence} holds for some $L\geq 0$, $R_K>0$, $\lambda\geq $, and $\metricRhoY\geq 0$, $\tilde{\gamma}_G>0$, $\tilde{\gamma}_{F^*} \defeq \min\{\gamma_{\LIN},\gamma_{\NL}-(p-1)\zeta\}>0$, and $\theta\ge\metricRhoY^{2-p}p^{-p}\omega^{-1}\zeta^{1-p}$ for some $\zeta>0$ and $p\in[1,2]$. Choose 
    \begin{equation}
        \label{eq:linear-rate-rules}
        0 < \tau_i\equiv \tau\le\min\biggl\{\textstyle\frac{\delta}{3L\metricRhoY+\lambda},
        \frac{\sqrt{(1-\kappa)\tilde{\gamma}_{F^*}/\tilde{\gamma}_G}}{R_K}\biggr\},
        \quad
        \sigma_i\equiv \sigma \defeq \frac{\tilde{\gamma}_{G}}{\tilde{\gamma}_{F^*}}\tau,
        \quad
        \omega_i\equiv \omega \defeq \frac{1}{1+2\tilde{\gamma}_{G}\tau}
    \end{equation}
    for some $0\le\delta\le\kappa<1$. 
    Then $\norm{\thisu-\realoptu}^2$ converges to zero at the rate $O(\omega^N)$.
\end{theorem}
\begin{proof}
    The proof follows that of \cref{thm:acceleration-nlpdhgm}. 
    To verify \eqref{eq:scalar-rules}, we take $\sigmaTest_{0}\defeq 1/\sigma$, and $\tauTest_0\defeq 1/\tau$, for which \eqref{eq:scalar-test-update} is satisfied due to the second relation  of \eqref{eq:linear-rate-rules}. By induction, we further obtain from this
    \begin{equation}
        \label{eq:scalar-linear-test-update}
        \tauTest_i\tau=\sigmaTest_{i}\sigma=(1+2\tilde{\gamma}_{G}\tau)^i
    \end{equation}
    for all $i\in\N$, verifying \eqref{eq:scalar-step-rules0}.
    Inequality \eqref{eq:scalar-gamma-rules} holds due to $\tilde{\gamma}_{F^*}=\min\{\gamma_{\LIN},\gamma_{\NL}-(p-1)\zeta\}>0$.
    It remains to prove \eqref{eq:scalar-step-rules-test}, which follow via \eqref{eq:scalar-tau-sigma-estimate-rhoy} from the bound on $\tau$ in \eqref{eq:linear-rate-rules}.
    Finally, we apply \cref{lemma:zimi-estim}, \eqref{eq:scalar-linear-test-update} for $i=N$, and \cref{thm:convergence-result-main-h} to conclude that
    \[
        (1+2\tilde{\gamma}_{G}\tau)^N\left(\frac{\delta}{2\tau}\norm{x^N-\realoptx}^2+\frac{\kappa-\delta}{2\sigma(1-\delta)}\norm{y^N-\realopty}^2\right)
        \le
        \frac{1}{2}\norm{u^0-\realoptu}^2_{\Test_{1}\Precond_{1}},
    \]
    which yields the desired convergence rate.
\end{proof}

\begin{remark}[global convergence]
    \label{rmk:global-convergence}
    Following \cref{rmk:no-locality}, suppose that \cref{ass:main-for-convergence}\,\ref{item:main-for-convergence-lipschitz-gradk}--\ref{item:main-for-convergence:k-nonlinear} hold for $\neighx_K=\neighx_G=X$ and that $\Dom F^*$ is bounded. If we then take $\metricRhoY$ large enough that $\Dom F^*\subseteq  \B_\NL(\realopty, \metricRhoY)$, \cref{ass:main-for-convergence}\,\ref{item:main-for-convergence:neighbourhood} will be satisfied for \emph{any} choice of starting point $u^0=(x^0,y^0)\in X\times Y$, i.e., we have global convergence.
    Note, however, that in this case we need $\grad K$ to be bounded on the whole space, i.e., $R_K=\sup_{x\in X}\norm{\grad K(x)}<\infty$ has to hold.
\end{remark}

\subsection{Neighborhood-compatible iterations}
\label{sec:neighborhood-compatible-iterations}

To conclude our analysis, we provide explicit conditions on the initialization of \cref{alg:NL-PDHGM} to ensure that \cref{ass:main-for-convergence}\,\ref{item:main-for-convergence:neighbourhood} (neighborhood-compatible iterations) holds in cases where the global convergence of \cref{rmk:global-convergence} cannot be guaranteed.

To begin, the following result shows that the rules \eqref{eq:scalar-rules} are consistent with the sequence $\{\thisu\}_{i \in \N}$ generated by \eqref{eq:ppext} remaining in $\neighu(\metricRhoY)$, provided that $\tau_0$ is sufficiently small and that the starting point $u^0=(x^0, y^0)$ is sufficiently far inside the interior of $\neighu(\metricRhoY)$.
\begin{lemma}
	\label{lemma:neighborhood-compatible-iterations}
	Let $0 \le \delta \le \kappa < 1$ and $\metricRhoY>0$ be given, and 
	assume \eqref{eq:scalar-rules} holds with
        \begin{equation}
            1/\sqrt{1+2\tau_i\tilde{\gamma}_{G}}\le\omega_i\le\omega_{i+1}\le1\quad (i\in\N).
        \end{equation}
        Assume further that $\sup_{x\in\neighx_K}\norm{\grad K(x)}\le R_K$.
	Define
        \begin{equation}\label{eq:scalar-step-length-rmax}
            r_{\max}\defeq \sqrt{2\inv\delta(\norm{x^0-\realoptx}^2+\inv\mu\norm{y^0-\realopty}^2)}
            \quad\text{with}\quad
            \mu\defeq\sigma_1\omega_0/\tau_0.
        \end{equation}
	Assume also that $\B(\realoptx,r_{\max}+\delta_x)\times\B(\realopty,\localRhoY+\delta_y)\subseteq \neighu(\metricRhoY)$ for some $\delta_x, \delta_y>0$ as well as $\localRhoY\ge r_{\max}\sqrt{\mu(1-\delta)\delta/(\kappa-\delta)}$. If the initial primal step length $\tau_0$ satisfies
	\begin{equation}
		\label{eq:scalar-step-length-bounds-r}
		\tau_0\le\min\left\{
		\frac{\delta_x}{2R_K\localRhoY +2L\norm{\Pnl \realopty}r_{\max}},
		\frac{2\delta_y\omega_0(r_{\max}+\delta_x)^{-1}}{(L(r_{\max}+\delta_x)+2R_K)\mu}\right\},
	\end{equation}
	then \cref{ass:main-for-convergence}\,\ref{item:main-for-convergence:neighbourhood} (neighborhood-compatible iterations) holds.
\end{lemma}
\begin{proof}
    We first set up some basic set inclusions. Without loss of generality, we can assume that $\sigmaTest_1 = 1$, as we can always rescale the testing variables $\tauTest_i$ and $\sigmaTest_i$ by the same constant without violating \eqref{eq:scalar-rules}. 
    We then define  $\localRhoX[i] \defeq\norm{u^0-\realoptu}_{\Test_{1}\Precond_{1}}/\sqrt{\delta\tauTest_i}$, $\delta_{x,i}\defeq\sqrt{\tauTest_0/\tauTest_i}\delta_x$, and
    \[
        \neighu_i \defeq\bigl\{(x,y) \in X \times Y \,\bigm|\, \norm{x-\realoptx}^2+{\textstyle \frac{\sigmaTest_{i+1}}{\tauTest_{i}}		\frac{\kappa-\delta}{(1-\delta)\delta}}\norm{y-\realopty}^2\le\localRhoX[i]^2 \bigr\}.
    \]
    We then observe from \cref{lemma:zimi-estim} that
    \begin{equation}
        \label{eq:neighui-zimi-rxi-inclusion}
        \{u \in X \times Y \mid \norm{u-\realoptu}_{\Test_{i+1}\Precond_{i+1}} \le \norm{u^0-\realoptu}_{\Test_{1}\Precond_{1}} \} \subset \neighu_i.
    \end{equation}
    From \eqref{eq:scalar-rules}, we also deduce that $\tauTest_{i+1}\ge\tauTest_{i}$ and hence that $\localRhoX[i+1]\le\localRhoX[i]$ as well as $\delta_{x,i}\le\delta_x$.
    Consequently, if $\localRhoX[0] \le r_{\max}$, then
    \begin{equation}
        \label{eq:ball-neigh-neghui-inclusion}
        \B(\realoptx,\localRhoX[i]+\delta_{x,i})\times \B(\realopty, \localRhoY+\delta_y)\subseteq \B(\realoptx, r_{\max}+\delta_x) \times \B(\realopty, \localRhoY+\delta_y)\subseteq\neighu(\metricRhoY),
    \end{equation}
    so it will suffice to show that $u^i \in \B(\realoptx,\localRhoX[i]+\delta_{x,i})\times \B(\realopty, \localRhoY+\delta_y)$ for each $i \in \N$ to prove the claim.
    We do this in two steps. The first step shows that $\localRhoX[i] \le r_{\max}$ and
    \begin{equation}
        \label{eq:neighui-product-inclusion}
        \neighu_i\subseteq\B(\realoptx,\localRhoX[i])\times \B(\realopty, \localRhoY)
        \quad (i \in \N).
    \end{equation}
    In the second step, we then show that $u^i \in \neighu_i$ as well as $\overnextx\in \neighx_K$ for $i \in \N$.
    The two inclusion  \eqref{eq:ball-neigh-neghui-inclusion} and \eqref{eq:neighui-product-inclusion} then imply that \cref{ass:main-for-convergence}\,\ref{item:main-for-convergence:neighbourhood} holds.

    \paragraph{Step 1}
    We first prove \eqref{eq:neighui-product-inclusion}.
    Since $\neighu_i\subseteq\B(\realoptx,\localRhoX[i])\times Y$, we only have to show that $\neighu_i\subseteq X\times \B(\realopty, \localRhoY)$. First, note that \eqref{eq:scalar-rules} implies that $\tilde\gamma_{F^*}\ge0$ and therefore $\sigmaTest_{i+1}\ge \sigmaTest_{i}\ge \sigmaTest_{1}=1$ as well as $\tauTest_{i+1}\geq\tauTest_i \geq \tauTest_0=\eta_1\omega_0/\tau_0=\mu$. We then obtain from the definition of $\localRhoX[i]$ that
    \[
        \localRhoX[i]^2\delta\tauTest_i
        =\norm{u^0-\realoptu}^2_{\Test_{1}\Precond_{1}}
        =\mu\norm{x^0-\realoptx}^2-2\eta_0\iprod{x^0-\realoptx}{\kgradconj{x^0}(y^0-\realopty)}+\norm{y^0-\realopty}^2.
    \]
    Using Cauchy's inequality, the fact that $\tauTest_i\geq \mu$, and the assumption $\norm{\grad K(x^0)}\le R_K$, we arrive at
    \[
        \localRhoX[i]^2
        \le(2\mu\norm{x^0-\realoptx}^2+(1+\eta_0^2\tauTest_0^{-1}R_K^2)\norm{y^0-\realopty}^2)/(\delta\mu).
    \]
    We obtain from \eqref{eq:scalar-step-rules-test} that $\eta_0^2\tauTest_0^{-1}R_K^2\le1-\kappa\le1$ and hence that $\localRhoX[i]^2\le r_{\max}^2$. The assumption on $\localRhoY$ then yields that
    \begin{equation}
        \label{eq:scalar-step-length-bounds-u0-dual}
        \localRhoY^2 
        \ge r_{\max}^2\tauTest_{0}\frac{(1-\delta)\delta}{\kappa-\delta}
        \ge\frac{\localRhoX[0]^2\tauTest_{0}}{\sigmaTest_{i+1}}\frac{(1-\delta)\delta}{\kappa-\delta}
        =\frac{\localRhoX[i]^2\tauTest_{i}}{\sigmaTest_{i+1}}\frac{(1-\delta)\delta}{\kappa-\delta}
    \end{equation}
    for all $i\in \N$, and \eqref{eq:neighui-product-inclusion} follows from the definition of $\neighu_i$.

    \paragraph{Step 2}
    We next show by induction that $\thisu \in \neighu_i$, $\overnextx\in \neighx_K$, and
    \begin{equation}
        \label{eq:scalar-step-length-bounds-ri}
        \tau_i\le\frac{\delta_{x,i}}{2R_K\localRhoY +2L\norm{\Pnl \realopty}\localRhoX[i]},
        \quad\text{and}\quad
        \sigma_{i+1}(\localRhoX[i]+\delta_{x,i})\le\frac{2\delta_y}{L(\localRhoX[i]+\delta_{x,i})+2R_K}
    \end{equation}	
    hold for all $i \in \N$.

    Since \eqref{eq:scalar-rules} holds, we can apply \cref{lemma:zimi-estim} to $\norm{u^0-\realoptu}_{\Test_{1}\Precond_{1}}$ to verify that $u^0\in\neighu_0$. 
    Moreover, since $\sigma_1=\mu\tau_0/\omega_0$, the bound \eqref{eq:scalar-step-length-bounds-ri} for $i=0$ follows from \eqref{eq:scalar-step-length-bounds-r} and the bound $\localRhoX[0]\le r_{\max}$ from Step 1. This gives the induction basis.

    Suppose now that $u^N \in \neighu_N$ and that \eqref{eq:scalar-step-length-bounds-ri} holds for $i=N$. By \eqref{eq:neighui-product-inclusion}, we have that $u^N\in \B(\realoptx,\localRhoX[N])\times \B(\realopty, \localRhoY)$.
    Since \eqref{eq:scalar-step-length-bounds-ri} guarantees \eqref{eq:step-length-bounds-r}, we can apply \cref{lemma:step-length-bounds-r} to obtain 
    \[
        u^{N+1}\in \B(\realoptx,\localRhoX[N]+\delta_{x,N})\times \B(\realopty, \localRhoY+\delta_y)\quad\text{and}\quad \bar{x}^{N+1}\in \B(\realoptx,\localRhoX[N]+\delta_{x,N}). 
    \]
    Together with \eqref{eq:ball-neigh-neghui-inclusion}, we obtain that $u^{N+1}\in \neighu(\metricRhoY)$ and $\bar{x}^{N+1}\in \neighx_K$.
    \Cref{thm:nonneg-penalty-nlpdhgm,thm:convergence-result-main-h} now imply that \eqref{eq:convergence-result-main-h} is satisfied for $i\le N$ with $\Penalty_{N+1}\le0$, which together with \eqref{eq:convergence-result-main-h} and \eqref{eq:neighui-zimi-rxi-inclusion} yields that $u^{N+1}\in\neighu_{N+1}$. This is the first part of the claim.
    To show \eqref{eq:scalar-step-length-bounds-ri}, we deduce from \eqref{eq:scalar-rules} that
    \begin{align}
        \tau_{N+1}&=\frac{\tau_{N}\tauTest_{N}}{\tauTest_{N+1}\omega_N}
        =\frac{\tau_{N}}{\omega_{N}(1+2\tau_N\tilde{\gamma}_{G})},
        &
        \localRhoX[N+1]&=\frac{\localRhoX[N]}{\sqrt{1+2\tau_N\tilde{\gamma}_{G}}},
        \\
        \sigma_{N+2}&=\frac{\sigma_{N+1}\sigmaTest_{N+1}}{\sigmaTest_{N+2}\omega_{N+1}}
        =\frac{\sigma_{N+1}}{\omega_{N+1}(1+2\sigma_{N+1}\tilde{\gamma}_{F^*})},
        &
        \delta_{x,N+1}&=\frac{\delta_{x,N}}{\sqrt{1+2\tau_N\tilde{\gamma}_{G}}}.
    \end{align}
    Hence, using $\omega_{N+1}\ge\omega_N$, $\omega_N\sqrt{1+2\tau_N\tilde{\gamma}_{G}}\ge1$, and $\localRhoX[N+1]\le\localRhoX[N]$, as well as the inductive assumption \eqref{eq:scalar-step-length-bounds-ri} shows that
    \begin{align*}
        &\tau_{N+1}
        =\frac{\delta_{x,N+1}}{\delta_{x,N}}\frac{\tau_N}{\omega_N\sqrt{1+2\tau_N\tilde{\gamma}_{G}}}
        \le
        \frac{\delta_{x,N+1}}{2R_K\localRhoY +2L\norm{\Pnl \realopty}\localRhoX[N+1]},
        \quad\text{and}
        \\
        &\sigma_{N+2}(\localRhoX[N+1]+\delta_{x,N+1})
        \le\frac{\sigma_{N+1}(\localRhoX[N]+\delta_{x,N})}{\omega_{N}{\sqrt{1+2\tau_N\tilde{\gamma}_{G}}}}
        \le\frac{2\delta_y}{L(\localRhoX[N+2]+\delta_{x,N+2})+2R_K}.
    \end{align*}
    This completes the induction step and hence the proof.
\end{proof}

If the step lengths and the over-relaxation parameter $\omega_i$ are constant, we can remove the lower bound on $\omega_i$ in \cref{lemma:neighborhood-compatible-iterations}.
\begin{lemma}
    The claims of \cref{lemma:neighborhood-compatible-iterations} also hold for $\tau_i\equiv\tau_0$, $\sigma_i\equiv\sigma_1$, and any choice of $\omega_i\equiv\omega\le1$. In particular, $\omega$ can be chosen according to \eqref{eq:linear-rate-rules}.
\end{lemma}
\begin{proof}
    The proof proceeds exactly as that of \cref{lemma:neighborhood-compatible-iterations}, replacing $\localRhoX[i]$ by $\localRhoX[0]$ and $\delta_{x,i}$ by $\delta_x$ everywhere.
	Since $r_{x,i} \le r_{x,0}$ and $\delta_{x,i}\le\delta_{x,0}$, the bound \eqref{eq:neighui-zimi-rxi-inclusion} holds in this case as well, while \eqref{eq:neighui-product-inclusion} reduces to the case $i=0$ that was shown in Step 1.
	Observe also that the only place where we needed the lower bound on $\omega_i$ was in the proof of \eqref{eq:scalar-step-length-bounds-ri} as part of the inductive step of Step~2.
	As the bound and the step lengths remain unchanged between iterations in this case, this is sufficient to conclude the proof, and the lower bound on $\omega_i$ is thus not required.
\end{proof}

Finally, as long as we start close enough to a solution, no additional step length bounds are needed to guarantee \cref{ass:main-for-convergence}\,\ref{item:main-for-convergence:neighbourhood} .
\begin{proposition}
    Under the assumptions of \cref{thm:weak-convergence-nlpdhgm}, \ref{thm:acceleration-nlpdhgm}, or \ref{thm:linear-convergence-nlpdhgm}, suppose that $\metricRhoY>0$.
    Then there exists an $\varepsilon>0$ such that for all $u^0=(x^0,y^0)$ satisfying
    \begin{equation}\label{eq:neighborhood-start-close}
        \sqrt{2\inv\delta(\norm{x^0-\realoptx}^2+\inv\mu\norm{y^0-\realopty}^2)}
        \le\varepsilon
        \quad \text{with}\quad
        \mu\defeq\sigma_1\omega_0/\tau_0,
    \end{equation}
    \cref{ass:main-for-convergence}\,\ref{item:main-for-convergence:neighbourhood} (neighborhood-compatible iterations) holds.
\end{proposition}
\begin{proof}
    For given $\varepsilon>0$, set $\localRhoY= \varepsilon\sqrt{\mu(1-\delta)\delta/(\kappa-\delta)}$ as well as $\delta_x=\sqrt{\varepsilon}$ and $\delta_y=\metricRhoY-\localRhoY$. Then for $\varepsilon>0$ sufficiently small, both $\delta_y>0$ and $\B(\realoptx,r_{\max}+\delta_x)\times\B(\realopty,\localRhoY+\delta_y)\subseteq \neighu(\metricRhoY)$ hold. Furthermore, \eqref{eq:neighborhood-start-close} yields that $r_{\max}\le \varepsilon$ in \cref{lemma:neighborhood-compatible-iterations}. Since
    \[
        \min\left\{
            \frac{\varepsilon^{-1/2}}{2R_K\sqrt{\mu(1-\delta)\delta/(\kappa-\delta)} +2L\norm{\Pnl \realopty}},
            \frac{2(\localRhoY-\varepsilon)\omega_0\varepsilon^{-1/2}}
        {(L(\varepsilon+\sqrt{\varepsilon})+2R_K)(\sqrt{\varepsilon}+1)\mu}\right\} \to\infty
    \]
    for $\varepsilon\to 0$, we can guarantee that \eqref{eq:scalar-step-length-bounds-r} holds for any given $\tau_0>0$ by further reducing $\varepsilon>0$. The claim then follows from \cref{lemma:neighborhood-compatible-iterations}.
\end{proof}

\section{Numerical examples}
\label{sec:examples}

We now illustrate the convergence and the effects of acceleration for a nontrivial example from PDE-constrained optimization. Following \cite{tuomov-pdex2nlpdhgm}, we consider as the nonlinear operator the mapping from a potential coefficient in an elliptic equation to the corresponding solution, i.e., for a Lipschitz domain $\Omega\subset\mathbb{R}^d$, $d\leq 3$ and $X=Y=L^2(\Omega)$, we set $S:x\mapsto z$ for $z$ satisfying
\begin{equation}\label{eq:example_pde}
    \left\{
        \begin{aligned}
            \Delta z + xz &= f\quad \text{on }\Omega,\\
        \partial_\nu z &= 0\quad\text{on }\partial\Omega.
    \end{aligned}
\right.
\end{equation}
Here $f\in L^2(\Omega)$ is given; for our examples below we take $f\equiv 1$. The operator $S$ is uniformly bounded for all $x\geq \varepsilon>0$ almost everywhere as well as completely continuous and twice Fréchet differentiable with uniformly bounded derivatives. Furthermore, for any $h\in X$, the application $\nabla S(x)^*h$ of the adjoint Fréchet derivative can be computed by solving a similar elliptic equation; see \cite[Section 3]{tuomov-pdex2nlpdhgm}. For our numerical examples, we take $\Omega=(-1,1)$ and approximate $S$ by a standard finite element discretization on a uniform mesh with $1000$ elements with piecewise constant $x$ and piecewise linear $z$. We use the MATLAB codes accompanying \cite{tuomov-pdex2nlpdhgm} that can be downloaded from \cite{nlpdhgm-code}.

The first example is the $L^1$ fitting problem
\begin{equation}
    \label{eq:l1fit_problem}
    \min_{x\in L^2(\Omega)} \frac1\alpha \norm{S(x) - z^\delta}_{L^1} + \frac{1}2 \norm{x}_{L^2}^2,
\end{equation}
for some noisy data $z^\delta\in L^2(\Omega)$ and a regularization parameter $\alpha>0$; see \cite[Section 3.1]{tuomov-pdex2nlpdhgm} for details. For the purpose of this example, we take $z^\delta$ as arising from random-valued impulsive noise applied to $z^\dag = S(x^\dag)$ for $x^\dag(t) = 2-|t|$ and $\alpha = 10^{-2}$.
This fits into the framework of problem \eqref{eq:initial-problem} with $F(y) = \frac1\alpha \norm{y}_{L^1}$, $G(x) = \frac12\norm{x}_{L^2}^2$, and $K(x) = S(x)-z^\delta$. (Note that in contrast to \cite{tuomov-pdex2nlpdhgm}, we do not introduce a Moreau--Yosida regularization of $F$ here.) 
Due to the properties of $S$, the gradient of $K$ is uniformly bounded and Lipschitz continuous; cf.~\cite{Kroener:2009a}. Hence, following \cref{rmk:global-convergence}, we can expect that \cref{ass:main-for-convergence}\,\ref{item:main-for-convergence:neighbourhood} holds independent of the initialization. Furthermore, $G$ and $F^*$ are convex and hence \cref{ass:main-for-convergence}\,\ref{item:main-for-convergence:monotone-g-f} is satisfied for $\tilde \gamma_G,\tilde\gamma_{F^*}\geq 0$. This leaves \cref{ass:main-for-convergence}\,\ref{item:main-for-convergence:k-nonlinear}, which amounts to quadratic growth condition of \eqref{eq:l1fit_problem} near the minimizer; cf.~\cref{pr:strongly-convex-primal}.
Similar assumptions are needed for the convergence of Newton-type methods, see, e.g., \cite{Ulbrich2003}. In the context of PDE-constrained optimization, they are generally difficult to prove a priori and have to be assumed.
To set the initial step lengths, we estimate as in \cite{tuomov-pdex2nlpdhgm} the Lipschitz constant $L$ by $\tilde L=\max\{1,\norm{\grad S'(u^0)u^0}/\norm{u^0}\}\approx 1$. We then set  $\tau_0 = (4 \tilde L)^{-1}$ and $\sigma_0 = (2\tilde L)^{-1}$.  The starting points are chosen as $x_0 \equiv 1$ and $y^0\equiv 0$ (which are not close to the expected saddle point).
\pgfplotsset{every axis label/.append style={font=\scriptsize}}
\begin{figure}[t]
    \centering
    \begin{minipage}[t]{0.495\textwidth}
        \centering
        \begin{tikzpicture}

\begin{axis}[%
width=\linewidth,
xmode=log,
xmin=1,
xmax=10000,
xminorticks=true,
ymode=log,
ymin=1e-06,
ymax=1,
yminorticks=true,
axis x line*=bottom,
axis y line*=left,
legend style={legend pos=south west,legend cell align=left,align=left,draw=none,font=\scriptsize}
]

\addplot [color=BuGn-G,line width=1.5pt]
  table[row sep=crcr]{%
1	0.573219859641694\\
2	0.579341151963082\\
3	0.45807023168593\\
4	0.378337722037108\\
5	0.322033516084784\\
6	0.280569399944789\\
9.00000000000001	0.204422504077996\\
12	0.163553018310814\\
14	0.145758358669245\\
16	0.132457377480891\\
18	0.122234612472224\\
20	0.114202082619497\\
22	0.107774926963806\\
25	0.100301716777102\\
28	0.0946696276314342\\
31	0.0903180214307453\\
34	0.0868808469174649\\
38	0.0833042243971803\\
42	0.0805322096108204\\
47.0000000000001	0.0778280680314949\\
54.0000000000001	0.0749474618213937\\
65	0.0716039875747553\\
90.0000000000001	0.0661429525830619\\
104	0.0635762027788029\\
118	0.0611699050196558\\
132	0.0588782840667392\\
147	0.0565283564775529\\
163	0.0541297983963917\\
180	0.0516954702795934\\
198	0.0492395333620786\\
217	0.0467763598356367\\
237	0.044319926973077\\
258	0.0418834867217041\\
280	0.0394793850403779\\
304	0.0370197347247085\\
329	0.0346271811899865\\
356	0.0322245581317418\\
385	0.0298387144435605\\
415	0.0275667445759396\\
448	0.0252792798720413\\
483.000000000001	0.0230746363759022\\
521.000000000001	0.0209151330898618\\
562	0.0188312375136404\\
606	0.0168485534348606\\
655	0.01491374108707\\
709.000000000001	0.0130726483601854\\
776.000000000001	0.0111584491818232\\
866	0.00913233226397841\\
1159	0.00534710215777363\\
1224	0.00488339605181356\\
1292	0.00450558104622966\\
1369	0.00416594929155871\\
1445	0.00390255315664939\\
1502	0.0037540064004685\\
1582	0.00357848442084874\\
1676	0.00340737373490013\\
1732	0.00331794667159707\\
1794	0.0032305685242736\\
1866	0.00315131998087017\\
2008	0.00303944996534395\\
2230	0.00292408619387567\\
2440	0.00286148372185509\\
2892	0.00275958213762829\\
3125	0.00271059347428948\\
3301	0.00267832094360448\\
3631	0.00262122778169694\\
4389	0.00250311460994031\\
4593	0.00247494246383144\\
5201	0.00241400293958555\\
5727	0.00236407583278198\\
6143	0.00230787857986468\\
6482.00000000001	0.00225382852094592\\
6827	0.00219719971891412\\
7451.00000000001	0.00209820424729257\\
8531	0.00195978209557356\\
10000	0.00181787944348522\\
};
\addlegendentry{$\tilde \gamma_G=0$}

\addplot [color=BuGn-K,line width=1.5pt]
  table[row sep=crcr]{%
1	0.573219859641694\\
2	0.567957730854697\\
3	0.44686972387294\\
4	0.360903174904372\\
5	0.298886655467567\\
6	0.252660197082455\\
7	0.217324545769382\\
9.00000000000001	0.16804691376767\\
11	0.136682708720634\\
12	0.125346086023511\\
13	0.116116704542664\\
14	0.108575691694977\\
15	0.10239667616838\\
16	0.0973217112176248\\
17	0.0931445123197818\\
18	0.0896985137605377\\
19	0.0868481782494816\\
20	0.0844825512482459\\
21	0.0825103954102696\\
23	0.0794585623146765\\
25	0.0772342639077489\\
28	0.0747945520602019\\
36	0.0699929080429338\\
39	0.0682294990224124\\
42	0.0663932126652765\\
45	0.0644822560834896\\
48	0.0625042007738767\\
51	0.0604676642586254\\
54.0000000000001	0.0583802380965235\\
57	0.0562491444720714\\
60	0.0540823023626262\\
63.0000000000001	0.0518887346109936\\
66.0000000000001	0.0496782199479855\\
70	0.0467214131772778\\
74	0.0437734006896313\\
78.0000000000001	0.0408528624067936\\
82	0.0379759643264688\\
85.9999999999999	0.0351571924009402\\
90.0000000000001	0.0324099653536396\\
93.9999999999999	0.0297468580653447\\
98.0000000000001	0.0271795703191508\\
102	0.0247187847590729\\
107	0.021808907670622\\
112	0.0190998385062124\\
117	0.0166055673760403\\
122	0.0143350293776646\\
127	0.0122877456289149\\
133	0.0101251441559223\\
140	0.00800505756270352\\
148	0.00607950631632005\\
169	0.0031086655093641\\
174	0.0027419096440557\\
179	0.00247192283871401\\
183	0.0023151073412009\\
186	0.00222753161482737\\
189	0.00216265684068807\\
192	0.00211778783974684\\
195	0.00209014108459769\\
198	0.0020773372042159\\
201	0.00207703638371322\\
204	0.00208709112297294\\
207	0.00210532040455019\\
211	0.00213887387211937\\
219	0.00222103389708384\\
229	0.00232053983130176\\
236	0.00237696207780447\\
242	0.00241289102717155\\
248	0.00243666014843242\\
254	0.00244885511741627\\
259	0.00244986133221825\\
265	0.00244020955085983\\
271	0.00242059255744522\\
279	0.00238455497450094\\
312	0.00223945231515941\\
321	0.00221973771762887\\
329	0.00221154587556687\\
338	0.00221139699300507\\
348	0.00222004886297026\\
387	0.00226198337523513\\
397	0.00225715116785278\\
407	0.00224303421555257\\
417	0.00221995271696916\\
427	0.00218812955738892\\
438	0.00214416745244392\\
451	0.00208343160283719\\
471	0.00198342812115555\\
502	0.00184624353838891\\
521.000000000001	0.00178161264772156\\
541	0.00172810880966122\\
597	0.00160104517049961\\
618	0.00154758064519899\\
644.000000000001	0.00147706764600564\\
725	0.00128718221566489\\
764	0.00121095348735861\\
791.000000000001	0.00115552829157998\\
823	0.0010875328948926\\
1007	0.000784702803989675\\
1100	0.000664829026862442\\
1144	0.000613494766050936\\
1185	0.000565454779672184\\
1252	0.000491900121041399\\
1353	0.000400520593620503\\
1450	0.000333049282125937\\
1532	0.000291915094033921\\
1594	0.000267668598652445\\
1641	0.00025360966496745\\
1694	0.000241173436279668\\
1751	0.000230516725116066\\
1803	0.000222987018176223\\
1880	0.000214219354721458\\
2006	0.000202170287681936\\
2188	0.000188854786909501\\
2266	0.000183337089842525\\
2334	0.000177783537650441\\
2400	0.00017168416719606\\
2462	0.000165169789529987\\
2529	0.000157349229205067\\
2599	0.000148508041986636\\
2676	0.000138221762178926\\
2785	0.000123826314890133\\
3115	9.04718658852568e-05\\
3240	8.20705845075861e-05\\
3341	7.68329549883802e-05\\
3442	7.27521417916629e-05\\
3539	6.97109553236642e-05\\
3801	6.28270677565677e-05\\
3889	6.00267923960742e-05\\
3986	5.66147835690577e-05\\
4116.00000000001	5.18691552373298e-05\\
4589	3.8104245583872e-05\\
4688	3.64224032516815e-05\\
4769	3.54391953003203e-05\\
4849	3.47783659229117e-05\\
4940	3.42745948624697e-05\\
5270.00000000001	3.28105860808242e-05\\
5386	3.20109925871484e-05\\
5549.00000000001	3.07268980187994e-05\\
5710	2.93434307950535e-05\\
6042	2.67107801936862e-05\\
6156	2.61396818460218e-05\\
6269	2.57629964574335e-05\\
6522.00000000001	2.50338709428919e-05\\
6641	2.45357877720801e-05\\
6755	2.39201656466576e-05\\
6869	2.31483079900734e-05\\
6993.00000000001	2.21364803442233e-05\\
7130.00000000001	2.08450876116485e-05\\
7346	1.87031728698705e-05\\
7589	1.67164041248365e-05\\
7825.00000000001	1.52654304294317e-05\\
8627.00000000001	1.15566447691211e-05\\
8813	1.09058558918866e-05\\
8958	1.05523841395477e-05\\
9113.00000000001	1.02876880090004e-05\\
9263	1.01131794795012e-05\\
9673	9.72750048853898e-06\\
9796	9.51699519965747e-06\\
9954	9.16885560969923e-06\\
10000	9.05988125942033e-06\\
};
\addlegendentry{$\tilde\gamma_G=\frac12$}


\addplot [domain=600:6000,color=black, dashed, line width=0.5pt]
{300/x^2};
\addlegendentry{$\mathcal{O}(N^{-2})$}

\end{axis}
\end{tikzpicture}%
        \caption{$L^1$ fitting: $\|x^N-\hat x\|^2_{L^2}$ for different values of $\tilde \gamma_G$}
        \label{fig:l1fit:accel}
    \end{minipage}\hfill
    \begin{minipage}[t]{0.495\textwidth}
        \centering
        \input{l1fit_conv_linear_gammas.tikz}
        \caption{$L^1$ fitting: $\|u^N-\hat u\|^2_{L^2\times L^2}$ (solid) and bounds $(1+2\tilde\gamma_{G}\tau)^{-N}$ (dashed) for strongly convex $F^*$ and different values of $\tilde\gamma_{F^*}$}
        \label{fig:l1fit:linear}
    \end{minipage}
\end{figure}
\Cref{fig:l1fit:accel} shows the convergence behavior $\norm{x^N-\hat x}_{L^2}^2$ of the primal iterates for $N\in\{1,\dots,N_{\max}\}$ for $N_{\max} = 10^{4}$, both without and with acceleration. Since the exact minimizer to \eqref{eq:l1fit_problem} is unavailable, here we take $\hat x:=x^{2N_{\max}}$ as an approximation. As can be seen, the convergence in the first case (corresponding to $\tilde\gamma_G = 0$) is at best $\mathcal{O}(N^{-1})$, while the accelerated algorithm according to \cref{thm:acceleration-nlpdhgm} with $\tilde\gamma_G = \frac12 < \gamma_G$ indeed eventually enters a region where the rate is $\mathcal{O}(N^{-2})$. 
If we replace $F$ by its Moreau--Yosida regularization $F_\gamma$, i.e., replace $F^*$ by $F^*_\gamma := F^* + \frac\gamma2\norm{\cdot}_Y^2$, \cref{thm:linear-convergence-nlpdhgm} is applicable for $\tilde \gamma_{F^*} = \gamma>0$.
As \cref{fig:l1fit:linear} shows for different choices of $\gamma$ and constant step sizes $\tau = \sqrt{{\tilde\gamma_{F^*}}/{\tilde \gamma_G}}\tilde L^{-1}$, $\sigma = ({\tilde \gamma_G}/{\tilde\gamma_{F^*}})\tau$,
the corresponding algorithm leads to linear convergence of the full iterates $\norm{u^N-\hat u}_{L^2\times L^2}^2$ with a rate of $(1+2\tilde \gamma_{G}\tau)^{-N}$ (which depends on $\gamma$ by way of $\tau$).

We also consider the example of optimal control with state constraints mentioned in the Introduction, i.e.,
\begin{equation}\label{eq:state_problem}
    \min_{x\in L^2}  \frac1{2\alpha}\norm{S(x)-z^d}_{L^2}^2 + \frac{1}2 \norm{x}_{L^2}^2
    \quad\text{s.\,t.}\quad [S(x)](t) \leq c  \quad\text{a.\,e. in } \Omega,
\end{equation}
see \cite[Section 3.3]{tuomov-pdex2nlpdhgm} for details.
Here we choose $z^d=S(x^\dag)$ with $x^\dag$ as above, $\alpha = 10^{-3}$, and $c=0.68$ such that the state constraints are violated for $z^d$. 
Again, this fits into the framework of problem \eqref{eq:initial-problem} with $F(y) = \frac1{2\alpha}\norm{y-z^d}_{L^2}^2 + \delta_{(-\infty,c]}(y)$, $G(x) = \frac12\norm{x}_{L^2}^2$, and $K(x) = S(x)$. With the same parameter choice as in the last example, we again eventually observe the convergence rate of $\mathcal{O}(N^{-2})$ for the accelerated algorithm (see \cref{fig:state:accel}) as well as linear convergence if the state constraints are replaced by a Moreau--Yosida regularization (see \cref{fig:state:linear}).
\begin{figure}[t]
    \centering
    \begin{minipage}[t]{0.495\textwidth}
        \centering
        \begin{tikzpicture}

\begin{axis}[%
width=\linewidth,
xmode=log,
xmin=1,
xmax=10000,
xminorticks=true,
ymode=log,
ymin=1e-08,
ymax=1,
yminorticks=true,
axis x line*=bottom,
axis y line*=left,
legend style={legend pos=south west,legend cell align=left,align=left,draw=none,font=\scriptsize}
]

\addplot [color=BuGn-G,line width=1.5pt]
  table[row sep=crcr]{%
1	0.509885342569757\\
2	0.515405627025384\\
3	0.394867842151682\\
4	0.315841937558894\\
5	0.260160132306255\\
6	0.219251326057506\\
7	0.188165284535725\\
8	0.163885351089645\\
9.00000000000001	0.144489313301029\\
11	0.115657420004514\\
13	0.095486253426641\\
23	0.0492065547903402\\
26	0.0430797458584373\\
29	0.0385366221685455\\
32	0.0351007544171103\\
35	0.0324578832648923\\
38	0.0303940103207044\\
41	0.0287596129162199\\
44	0.027447990376543\\
47.0000000000001	0.0263816547780531\\
51	0.0252448944455368\\
55	0.024347139542332\\
60	0.0234610581757181\\
66.0000000000001	0.0226323422027121\\
74	0.0217761296819449\\
89.0000000000001	0.0205693682695727\\
109	0.0193027384656465\\
124	0.0184621431611711\\
139	0.0176756508274749\\
155	0.016883345034978\\
172	0.0160880005367786\\
190	0.0152942326245928\\
209	0.0145071384587435\\
230	0.013694266561019\\
253	0.0128682878120826\\
278	0.0120412177984564\\
305	0.0112241137639609\\
335	0.0104007546318459\\
368	0.00958783877028228\\
406.000000000001	0.00875903510108287\\
451	0.00790759012449992\\
507.000000000001	0.00701589016576563\\
590.000000000001	0.00597160335627852\\
754	0.00459972808769839\\
835	0.00415279642787215\\
910.000000000001	0.00383116767493943\\
984.000000000001	0.00357985260743933\\
1059	0.00337630929671193\\
1138	0.00320425920629549\\
1224	0.00305394763644002\\
1322	0.00291687493396292\\
1439	0.00278643848796572\\
1594	0.00264947945465977\\
1995	0.00238694882261878\\
2273	0.00223933378378733\\
2539	0.00211183435174905\\
2827	0.00198588929839687\\
3158	0.0018552538526998\\
3560	0.00171530138981922\\
4080.00000000001	0.0015611411605909\\
4815	0.00138531750545909\\
6041	0.00116999405911406\\
9329	0.000841461831156727\\
10000	0.000798088391631224\\
};
\addlegendentry{$\tilde \gamma_G=0$}

\addplot [color=BuGn-K,line width=1.5pt]
  table[row sep=crcr]{%
1	0.509885342569757\\
2	0.504098048910511\\
3	0.383747215826097\\
4	0.298561774975139\\
5	0.237261752238444\\
6	0.191695646771497\\
7	0.15697517513138\\
8	0.130027296765462\\
9.00000000000001	0.108825077375103\\
10	0.0919754986885958\\
11	0.078487584490149\\
16	0.0414109123192656\\
17	0.0376565500959638\\
18	0.0346249622807809\\
19	0.0321815969356242\\
20	0.0302164072942735\\
21	0.0286390026870399\\
22	0.0273749276472894\\
23	0.0263627638115519\\
24	0.0255518401455288\\
25	0.0249004001275162\\
26	0.0243741170256474\\
28	0.0235897763221141\\
30	0.0230314914213106\\
35	0.0220152430181166\\
38	0.0214465507652415\\
41	0.0208474485840827\\
44	0.0202153546323628\\
47.0000000000001	0.0195591570765671\\
50	0.0188879983746931\\
53	0.0182078451359842\\
56	0.0175219641750128\\
59.0000000000001	0.0168325343657215\\
63.0000000000001	0.0159118594515668\\
67.0000000000001	0.0149959640631265\\
71	0.0140923995515914\\
75	0.0132079966488961\\
79.0000000000001	0.0123483780748649\\
83	0.0115182136442939\\
88.0000000000001	0.0105280772593773\\
93.0000000000001	0.00959713725093964\\
98.9999999999999	0.00856576122164085\\
105	0.0076340652548312\\
113	0.00655264709857767\\
135	0.00449808267187475\\
141	0.00414583249625756\\
147	0.00386711826068564\\
152	0.0036840637840086\\
157	0.00353963880800761\\
162	0.00342814639681739\\
167	0.00334404361498715\\
172	0.00328205261514094\\
177	0.00323725527444572\\
183	0.00319991514982091\\
191	0.00316701594456281\\
210	0.00310134906930605\\
218	0.00306035572002298\\
225	0.00301365844069982\\
232	0.0029563446880845\\
239	0.00288896195636784\\
247	0.00280126189450006\\
255	0.00270465085969641\\
265	0.00257581978299594\\
276	0.00243003355001378\\
291	0.00223547795705267\\
322	0.00188677051426468\\
352.000000000001	0.00163170110038053\\
386	0.00141559427546042\\
462.000000000001	0.00108287478779593\\
550	0.000840990150832062\\
634	0.000682241000397244\\
730	0.000550168817722156\\
854	0.000429937007271574\\
978.000000000001	0.000344963348568018\\
1106	0.000280381691580653\\
1245	0.000227889786068797\\
1395	0.000185237203045173\\
1555	0.000150742890173064\\
1725	0.000122758873162005\\
1906	9.98831853591333e-05\\
2099	8.10887532952577e-05\\
2303	6.57335014314317e-05\\
2517	5.32338498267646e-05\\
2742	4.30018829528263e-05\\
2977	3.46604341794071e-05\\
3223	2.78322726160915e-05\\
3478	2.22885009209119e-05\\
3743	1.77710720863258e-05\\
4016	1.41194694909676e-05\\
4297	1.1168453016305e-05\\
4587	8.77956594466934e-06\\
4885	6.85809967195816e-06\\
5191.00000000001	5.31826169954988e-06\\
5510	4.07365919087674e-06\\
5847.00000000001	3.06618059135992e-06\\
6191	2.28689721189565e-06\\
6563	1.66062980268385e-06\\
6985.00000000001	1.15276265127652e-06\\
7552	7.11616356988365e-07\\
8165.00000000001	4.41937038059473e-07\\
8480.00000000001	3.59249063879328e-07\\
8728.00000000001	3.13294276891405e-07\\
8940	2.8478533581183e-07\\
9129.00000000001	2.66190376819773e-07\\
9300.00000000001	2.53966605733296e-07\\
9458.00000000001	2.45930945829879e-07\\
9605.00000000001	2.40828455124308e-07\\
9744	2.37786797793711e-07\\
9879.00000000001	2.36237758280132e-07\\
10000	2.35846034288777e-07\\
};
\addlegendentry{$\tilde \gamma_G=\frac12$}


\addplot [domain=300:3000,color=black, dashed, line width=0.5pt]
{100/x^2};
\addlegendentry{$\mathcal{O}(N^{-2})$}

\end{axis}
\end{tikzpicture}%
        \caption{State constraints: $\|x^N-\hat x\|^2_{L^2}$ for different values of $\tilde \gamma_G$}
        \label{fig:state:accel}
    \end{minipage}\hfill
    \begin{minipage}[t]{0.495\textwidth}
        \centering
        \input{state_conv_linear_gammas.tikz}
        \caption{State constraints: $\|u^N-\hat u\|^2_{L^2\times L^2}$ (solid) and bounds $(1+2\tilde\gamma_{G}\tau)^{-N}$ (dashed) for strongly convex $F^*$ and different values of $\tilde\gamma_{F^*}$}
        \label{fig:state:linear}
    \end{minipage}
\end{figure}

\section{Conclusions}

We have developed sufficient conditions on primal and dual step lengths that ensure (in some cases global) convergence and higher convergence rates of the NL-PDHGM method for nonsmooth nonconvex optimization. We have proved that usual acceleration rules give local $O(1/N^2)$ convergence, justifying their use in previously published numerical examples \cite{tuomov-pdex2nlpdhgm}.
Furthermore, we have derived novel linear convergence results based on bounds on the initial step lengths. 
Since our main derivations hold for general operators, one potential extension of the present work is to combine its approach with that of \cite{tuomov-blockcp} to derive block-coordinate methods for nonconvex problems.

\section*{Acknowledgments}

T.~Valkonen and S.~Mazurenko have been supported by the EPSRC First Grant EP/P021298/1, ``PARTIAL Analysis of Relations in Tasks of Inversion for Algorithmic Leverage''.
C.~Clason is supported by the German Science Foundation (DFG) under grant Cl~487/1-1.

\section*{\texorpdfstring{\normalsize}{}A data statement for the EPSRC}

{\color{red}All data and source codes will be publicly deposited when the final accepted version of the manuscript is submitted.}

\appendix

\section{A small improvement of Opial's lemma}

The earliest version of the next lemma is contained in the proof of \cite[Theorem 1]{opial1967weak}.
\begin{lemma}[{\cite[Lemma 6]{browder1967convergence}}]
    \label{lemma:opial}
    On a Hilbert space $X$, let $\hat X \subset X$ be closed and convex, and $\{\thisx\}_{i \in \N} \subset X$. Then $\thisx \weakto \bar{x}$ weakly in $X$ for some $\bar{x} \in \hat X$ if:
    \begin{enumerate}[label=(\roman*)]
        \item\label{item:opial-non-increasing} $i \mapsto \norm{\thisx-\bar{x}}$ is nonincreasing for all $\bar{x} \in \hat X$.
        \item\label{item:opial-limit} All weak limit points of $\{\thisx\}_{i \in \N}$ belong to $\hat X$.
    \end{enumerate}
\end{lemma}
We can improve this result to the following
\begin{lemma}
    \label{lemma:opial-improved}
    Let $X$ be a Hilbert space, $\hat X \subset X$ (not necessarily closed or convex), and $\{\thisx\}_{i \in \N} \subset X$. Also let $A_i \in \linear(X; X)$ be self-adjoint and $A_i \ge \hat{\varepsilon}^2 I$ for some $\hat{\varepsilon}\ne0$ for all $i \in \N$. If the following conditions hold, then $\thisx \weakto \bar{x}$ weakly in $X$ for some $\bar{x} \in \hat X$:
    \begin{enumerate}[label=(\roman*)]
        \item\label{item:opial-improved-non-increasing} $i \mapsto \norm{\thisx-\hat x}_{A_i}$ is nonincreasing for \emph{some} $\hat x \in \hat X$.
        \item\label{item:opial-improved-limit} All weak limit points of $\{\thisx\}_{i \in \N}$ belong to $\hat X$.
        \item\label{item:opial-improved-a-limit} There exists $C$ such that $\norm{A_{i}}\le C^2$ for all $i$, and for any weakly convergent subsequence $x_{i_k}$ there exists $A_\infty  \in \linear(X; X)$ such that $A_{i_k}x \to A_\infty x$ strongly in $X$ for all $x \in X$.
    \end{enumerate}
\end{lemma}
\begin{proof}
    For $x \in \closure \conv \hat X$, define $p(x) \defeq \liminf_{i \to \infty} \norm{x-\thisx}_{A_i}$. Clearly \ref{item:opial-improved-non-increasing} yields
    \[
        p(\hat x) = \lim_{i \to \infty} \norm{\hat x-\thisx}_{A_i} \in [0, \infty).
    \]
    Using the triangle inequality and \ref{item:opial-improved-a-limit}, for any $x, x' \in \closure \conv \hat X$ moreover
    \begin{align}
        \label{eq:opial-improved-p-upper}
        0 \le p(x) & \le p(x') + \limsup_{i \to \infty} \norm{x' - x}_{A_i} \le p(x') + C \norm{x' - x}.
    \end{align}
    Choosing $x'=\hat x$ we see from \eqref{eq:opial-improved-p-upper} that $p$ is well-defined and finite.
    It is moreover bounded from below.
   	Given $\varepsilon>0$, we can therefore find $x_\varepsilon^* \in \closure \conv \hat X$ such that $p(x_\varepsilon^*)^2 - \varepsilon^2 \le \inf_{\closure\conv \hat X} p^2$.
   	The norm $\norm{x_\varepsilon^*}$ is bounded from above for small values of $\varepsilon$: for the subsequence $\{x_{i_k}\}$ realizing the limes inferior in $p(x_\varepsilon^*)$,
   	\[
	   	\norm{x_\varepsilon^*}_{A_{i_k}}\le \norm{x_\varepsilon^*-x^{i_k}}_{A_{i_k}}+\norm{x^{i_k}-\hat{x}}_{A_{i_k}}+\norm{\hat{x}}_{A_{i_k}},
   	\]
   	and consequently
   	\[
	   	\hat{\varepsilon}\norm{x_\varepsilon^*}\le
	   	\left(\inf_{\closure\conv \hat X} p\right) + \varepsilon +\norm{x^{0}-\hat{x}}_{A_{0}}+C\norm{\hat{x}},
	\]
   	so there is a subsequence of  $\norm{x_\varepsilon^*}$ weakly converging to some $\bar{x}$ when $\varepsilon \downto 0$.
   	Without loss of generality, by restricting the allowed values of $\varepsilon$, we may assume that $\bar{x}$ is unique.

    Let $\bar{x}'$ be some weak limit of $\{\thisx\}$.
    By \ref{item:opial-improved-limit}, $\bar{x}' \in \hat X$.
    We have to show that $\bar{x}=\bar{x}'$.
    For simplicity of notation, we may assume that the whole sequence $\{\thisx\}$ converges weakly to $\bar{x}'$. By \ref{item:opial-improved-a-limit}, for any $x \in X$, we have
    \begin{equation}
        \label{eq:opial-limit0}
        \lim_{i \to \infty} \iprod{x}{\bar{x}_\varepsilon-\thisx}_{A_i}
        =
        \lim_{i \to \infty} \left(
            \iprod{x}{\bar{x}_\varepsilon-\thisx}_{A_\infty }
            + \iprod{(A_i-A_\infty )x}{\bar{x}_\varepsilon-\thisx}
            \right)
        =
        \iprod{x}{\bar{x}_\varepsilon-\bar{x}'}_{A_\infty }.
    \end{equation}
    Moreover, for any $\lambda \in (0, 1)$, we have $\bar{x}_{\varepsilon,\lambda} \defeq (1-\lambda) \bar{x}_\varepsilon + \lambda \bar{x}' \in \closure\conv \hat X$.
    Now, since $\bar{x}$ is a minimizer of $p$ on $\closure\conv \hat X$, we can estimate
    \begin{equation}
        \label{eq:opial-inequality}
        \begin{aligned}[t]
            p(\bar{x}_\varepsilon)^2 - \varepsilon^2\le
            p(\bar{x}_{\varepsilon,\lambda})^2
            & =
                p(\bar{x}_\varepsilon)^2
                +
                \lim_{i \to \infty}\left(
                \lambda^2\norm{\bar{x}_\varepsilon-\bar{x}'}_{A_i}^2
                -2 \lambda \iprod{\bar{x}_\varepsilon-\bar{x}'}{\bar{x}_\varepsilon-\thisx}_{A_i}
                \right)
            \\
            & =
                p(\bar{x}_\varepsilon)^2
                +(\lambda^2-2\lambda)\norm{\bar{x}_\varepsilon-\bar{x}'}_{A_\infty }^2.
        \end{aligned}
    \end{equation}
    In the second equality we have used \ref{item:opial-improved-a-limit} and \eqref{eq:opial-limit0}. Now, since $\lambda^2 \le 2\lambda$, we obtain
    \[
    	0 \le (2\lambda-\lambda^2)\norm{\bar{x}_\varepsilon-\bar{x}'}^2_{A_\infty} \le \varepsilon^2.
    \]
    This implies $\bar{x}_\varepsilon \to \bar{x}'$ strongly as $\varepsilon \downto 0$. But also $\bar{x}­_\varepsilon \weakto \bar{x}$. Therefore $\bar{x}'=\bar{x}$.

    Finally, by $A_i \ge \hat{\varepsilon} I$ and \ref{item:opial-improved-non-increasing}, the sequence $\{\thisx\}$ is bounded, so any subsequence contains a weakly convergent subsequence. Since the limit is always $\bar{x}$, the whole sequence converges weakly to $\bar{x}$.
\end{proof}

\begin{remark}
    The condition $A_i \ge \hat{\varepsilon}^2 I$ is automatically satisfied if we replace \ref{item:opial-improved-a-limit} by $A_i \to A_\infty$ in the operator topology with $A_\infty \ge 2\hat{\varepsilon}^2 I$.
\end{remark}

\section{Reconstruction of the phase and amplitude of a complex number}
\label{sec:complex}

\def\ii{\mathbf{i}}
\def\angle{\upsilon}
\def\ampl{t}

The purpose of this appendix is to verify \cref{ass:k-nonlinear} for a simplified example related to the MRI reconstruction examples from \cite{tuomov-nlpdhgm}.
Consider
\[
    \min_{\ampl,\angle \in \R} \frac{1}{2}\abs{z-\ampl e^{\ii\angle}}^2 + G_0(\ampl),
    \quad\text{where}\quad
    G_0(\ampl) \defeq
    \begin{cases}
        \alpha \ampl, & \ampl \ge 0, \\
        \infty, & \ampl < 0.
    \end{cases}
\]
for some data $z \in \field{C}$ and a regularization parameter $\alpha>0$. We point out that the following does not depend on the specific structure of $G_0$ for $\ampl \ge 0$, as long as it is convex and increasing.
In terms of real variables, this can be written in general saddle point form as
\begin{equation}
\label{eq:complex-reconstr-h-def}
    K(\ampl, \angle) \defeq
    \begin{pmatrix}
        \ampl \cos \angle - \Re z \\
        \ampl \sin \angle - \Im z
    \end{pmatrix},
    \quad
    G(\ampl,\angle) \defeq G_0(\ampl),
    \quad
    \text{and}
    \quad
    F^*(\lambda, \mu) \defeq \frac{1}{2}(\lambda^2 + \mu^2).
\end{equation}
To simplify the notiation, let $x=(\ampl,\angle)$ and $y=(\lambda,\mu)$.

We now make a case distinction based on the sign of the optimal $\realopt\ampl\geq 0$. We first consider the case $\realopt\ampl>0$. 
\begin{lemma}
    \label{lemma:complex-reconstr-rho>0}
    Let $\realoptu \in \inv H(0)$, where $H(u)$ is defined in \eqref{eq:h} for $K$, $G$, and $F^*$ given by \eqref{eq:complex-reconstr-h-def}, and suppose $\realopt\ampl>0$.
    Let $L > \alpha\realopt\ampl/4$ as well as $\theta>0$ be arbitrary.
    Then \cref{ass:k-nonlinear} holds with $p=2$, i.e., there exists $\varepsilon>0$ such that for all $x, x' \in \B(\realoptx, \varepsilon)$,
    \begin{equation}\label{eq:complex-reconstr-K}
        \iprod{[\kgrad{x'}-\kgrad{\realoptx}]^*\realopty}{x-\realoptx}
        \ge
        \theta\norm{K(\realoptx)-K(x)-\kgrad{x}(\realoptx-x)}^2 - L(\angle-\angle')^2.
    \end{equation}
\end{lemma}
\begin{proof}
    The saddle point condition $0 \in H(\realoptu)$ expands as $K(\realopt{\ampl},\realopt{\angle}) \in \subdiff F^*(\realopt\lambda, \realopt\mu)$ and $-\kgradconj{\realopt{\ampl},\realopt{\angle}}\begin{psmallmatrix} \realopt\lambda \\ \realopt{\mu}\end{psmallmatrix} \in \subdiff G(\realopt{\ampl},\realopt{\angle})$.
    Since
    \[
        \kgrad{\ampl, \angle}
        =
        \begin{pmatrix}
            \cos \angle & -\ampl \sin \angle \\
            \sin \angle & \ampl \cos \angle
        \end{pmatrix}
        \quad
        \text{and}
        \quad
        \kgradconj{\ampl, \angle}\begin{pmatrix}\lambda \\ \mu \end{pmatrix}
        =
        \begin{pmatrix}
            \lambda \cos \angle + \mu \sin \angle \\
            \mu \ampl \cos \angle - \lambda\ampl \sin \angle
        \end{pmatrix},
    \]
    the latter further expands as
    \begin{equation}
    \label{eq:complex-foc}
        -(\realopt\lambda \cos \realopt\angle + \realopt\mu \sin \realopt\angle) \in \subdiff G_0(\realopt\ampl),
        \quad\text{and}\quad
        \realopt\mu \realopt\ampl \cos \realopt\angle = \realopt\lambda\realopt\ampl \sin \realopt\angle.
    \end{equation}
    From the second equality, $\realopt\mu \cos \realopt\angle= \realopt\lambda \sin \realopt\angle$.
    Since $\subdiff G_0(\ampl)=\alpha$ for $\ampl>0$, multiplying the first equality by $\cos \realopt\angle$ and $\sin \realopt\angle$ results in
    \begin{equation}
        \label{eq:complex-reconstr-rho>0-lambda-mu}
        \realopt\lambda = -\alpha \cos\realopt\angle,
        \quad
        \text{and}
        \quad
        \realopt\mu = -\alpha \sin\realopt\angle.
    \end{equation}
    We can thus write the left-hand side of \eqref{eq:complex-reconstr-K} as
    \begin{equation*}
        \begin{aligned}[t]
        d_1 &\defeq \iprod{[\kgrad{x'}-\kgrad{\realoptx}]^*\realopty}{x-\realoptx}
        \\
        &=
        \realopt\lambda(\cos\angle'-\cos\realopt\angle)(\ampl-\realopt\ampl)
        +\realopt\mu(\sin\angle'-\sin\realopt\angle)(\ampl-\realopt\ampl)
        \\ \MoveEqLeft[-1]
        +\realopt\lambda(-\ampl'\sin\angle'+\realopt\ampl\sin\realopt\angle)(\angle-\realopt\angle)
        +\realopt\mu(\ampl'\cos\angle'-\realopt\ampl\cos\realopt\angle)(\angle-\realopt\angle)\\
        &= -\alpha\bigl[\cos\realopt\angle(\cos\angle'-\cos\realopt\angle)(\ampl-\realopt\ampl)
            +\sin\realopt\angle(\sin\angle'-\sin\realopt\angle)(\ampl-\realopt\ampl)
            \\ \MoveEqLeft[-1]\qquad
            +\cos\realopt\angle(-\ampl'\sin\angle'+\realopt\ampl\sin\realopt\angle)(\angle-\realopt\angle)
            +\sin\realopt\angle(\ampl'\cos\angle'-\realopt\ampl\cos\realopt\angle)(\angle-\realopt\angle)\bigr]
            \\
            &=
            -\alpha\bigl[\cos\realopt\angle(\cos\angle'-\cos\realopt\angle)(\ampl-\realopt\ampl)
            +\sin\realopt\angle(\sin\angle'-\sin\realopt\angle)(\ampl-\realopt\ampl)
            \\ \MoveEqLeft[-1]\qquad
            +\ampl'[\sin\realopt\angle\cos\angle'-\cos\realopt\angle\sin\angle'](\angle-\realopt\angle)\bigr].
        \end{aligned} 
    \end{equation*}
    Using the standard trigonometric identities
    \begin{subequations}
    \label{eq:trigid}
    \begin{align}
        2\cos\realopt\angle\cos\angle'& =\cos(\realopt\angle-\angle')+\cos(\realopt\angle+\angle'),
        & 2\sin\realopt\angle\sin\angle'&=\cos(\realopt\angle-\angle')-\cos(\realopt\angle+\angle'), \\
        2\sin\realopt\angle\cos\angle'& =\sin(\realopt\angle+\angle')+\sin(\realopt\angle-\angle'),
        & 2\cos\realopt\angle\sin\angle'& =\sin(\realopt\angle+\angle')-\sin(\realopt\angle-\angle'),
    \end{align}
    \end{subequations}
    as well as $\cos^2\realopt\angle+\sin^2\realopt\angle=1$, this becomes
    \begin{equation*}
            d_1 =
            -\alpha\bigl[[\cos(\realopt\angle-\angle')-1](\ampl-\realopt\ampl)
            +\ampl'\sin(\realopt\angle-\angle')(\angle-\realopt\angle)
            \bigr].
    \end{equation*}
    Using Taylor expansion, we obtain for some $\eta_1$ and $\eta_2$ between $0$ and $\realopt\angle-\angle'$ that
    \begin{equation*}
        \begin{split}
            d_1 & =
            \alpha\left[\frac{\cos\eta_1}{2}(\realopt\angle-\angle')^2(\ampl-\realopt\ampl)
            -\ampl'\cos\eta_2(\realopt\angle-\angle')(\angle-\realopt\angle)
            \right]
            \\ &
            =
            \alpha\left[\frac{\cos\eta_1}{2}(\realopt\angle-\angle')^2(\ampl-\realopt\ampl)
            +\ampl'\cos\eta_2(\angle-\realopt\angle)^2 
            -\ampl'\cos\eta_2(\angle-\angle')(\angle-\realopt\angle)
            \right].
        \end{split}
    \end{equation*}
    Note that $\cos \eta_1,\cos\eta_2 \approx 1$ for $\angle'$ close to $\realopt\angle$.
    Using Cauchy's inequality, we have for any $\beta>0$ that
    \begin{equation}
        \label{eq:complex-reconstr-intermediary-d1}
        d_1 \ge
        \alpha\left[\frac{\cos\eta_1}{2}(\realopt\angle-\angle')^2(\ampl-\realopt\ampl)
        +(1-\beta)\ampl'\cos\eta_2(\angle-\realopt\angle)^2 
        -\frac{\cos\eta_2}{4\beta}\ampl'(\angle-\angle')^2
        \right].
    \end{equation}
    We also have $\frac{\cos\eta_1}{2}(\realopt\angle-\angle')^2(\ampl-\realopt\ampl) \ge -\abs{\ampl-\realopt\ampl}[(\angle-\realopt\angle)^2+(\angle-\angle')^2]$ and hence
    \begin{equation*}
        d_1 \ge
        \alpha\left[(1-\beta)\ampl'\cos\eta_2-\abs{\ampl-\realopt\ampl}\right](\angle-\realopt\angle)^2 
        -\alpha\left[\frac{\cos\eta_2}{4\beta}\ampl'+\abs{\ampl-\realopt\ampl}\right](\angle-\angle')^2.
    \end{equation*}
    Choosing $\varepsilon,\delta>0$ small enough, $\beta<1$ large enough, and $\ampl' \in \B(\realopt\ampl, \varepsilon)$, we can thus ensure that
    \begin{equation}
        \label{eq:complex-reconstr-final-d1}
        d_1 \ge
        \delta\theta(\angle-\realopt\angle)^2 
        -L(\angle-\angle')^2.
    \end{equation}

    We now turn to the right-hand side of \eqref{eq:complex-reconstr-K}, which we write as
    \begin{equation*}
        \begin{split}
            D_2 & \defeq
            K(\realoptx)-K(x)-\kgrad{x}(\realoptx-x)
            \\
            &
            =
            \begin{pmatrix}
                \realopt\ampl\cos\realopt\angle-\ampl\cos\angle-\cos\angle(\realopt\ampl-\ampl)+\ampl\sin\angle(\realopt\angle-\angle) \\
                \realopt\ampl\sin\realopt\angle-\ampl\sin\angle-\sin\angle(\realopt\ampl-\ampl)-\ampl\cos\angle(\realopt\angle-\angle)
            \end{pmatrix}
            \\
            &
            =
            \begin{pmatrix}
                \realopt\ampl(\cos\realopt\angle-\cos\angle)+\ampl\sin\angle(\realopt\angle-\angle) \\
                \realopt\ampl(\sin\realopt\angle-\sin\angle)-\ampl\cos\angle(\realopt\angle-\angle)
            \end{pmatrix}.
        \end{split}
    \end{equation*}
    Thus
    \begin{equation*}
        \begin{split}
            \norm{D_2}^2 & =
            2\realopt\ampl^2(1-\cos\realopt\angle\cos\angle-\sin\realopt\angle\sin\angle)
            +\ampl^2(\angle-\realopt\angle)^2
            \\ \MoveEqLeft[-1]
            +2\ampl\realopt\ampl(\realopt\angle-\angle)[
                (\cos\realopt\angle-\cos\angle)\sin\angle
                -(\sin\realopt\angle-\sin\angle)\cos\angle].
        \end{split}
    \end{equation*}
    Using the trigonometric identities \eqref{eq:trigid} and Taylor expansion, it follows that
    \begin{equation*}
        \begin{split}
            \norm{D_2}^2 & =
            2\realopt\ampl^2[1-\cos(\realopt\angle-\angle)]
            +\ampl^2(\angle-\realopt\angle)^2
            -2\ampl\realopt\ampl(\realopt\angle-\angle)\sin(\realopt\angle-\angle)
            \\
            &
            \le
            \realopt\ampl^2(\realopt\angle-\angle)^2
            +\ampl^2(\angle-\realopt\angle)^2
            -2\ampl\realopt\ampl(\realopt\angle-\angle)^2
            +2\ampl\realopt\ampl(\realopt\angle-\angle)^4
            \\
            &
            =
            (\realopt\ampl-\ampl)^2(\realopt\angle-\angle)^2
            +2\ampl\realopt\ampl(\realopt\angle-\angle)^4.
        \end{split}
    \end{equation*}
    By taking $\varepsilon>0$ small enough and $x=(\ampl,\angle) \in \B(\realoptx, \varepsilon)$, we thus obtain for    any $\delta>0$ that $\norm{D_2}^2 \le \delta (\realopt\angle-\angle)^2$.
    We now obtain from \eqref{eq:complex-reconstr-final-d1} that
    \[
        d_1 \ge \theta\norm{D_2}^2 - L(\angle-\angle')^2,
    \]
    which is exactly \eqref{eq:complex-reconstr-K}.
\end{proof}

The case of $\realopt\ampl=0$ is complicated by the fact that $\realopt\angle$ is then no longer unique. We therefore cannot expect convergence of $\thisx=(\this\ampl, \this\angle)$ in the sense studied in this work; we would instead need to consider convergence to the entire solution set; cf.~\cite{tuomov-subreg} for such an abstract approach for convex problems. However, under additional assumptions on the data $z$, we can proceed as before.
The next lemma lays the groundwork for showing that the algorithm actually converges locally to $\realopt\angle=\angle_z$ if $\ampl_z>0$. 
\begin{lemma}
    Suppose $\realopt\ampl=0$ and $\realopt\angle=\angle_z$ for $z=\ampl_z e^{\ii \angle_z}$ with $t_z>0$. Then the conclusions of \cref{lemma:complex-reconstr-rho>0} hold for some $\varepsilon>0$ and any $\theta, L > 0$.
\end{lemma}
\begin{proof}
    First we deduce from the optimality condition $K(\realopt{\ampl},\realopt{\angle}) \in \subdiff F^*(\realopt\lambda, \realopt\mu)$ that
    \[
        \realopt\lambda=-\Re z=-\ampl_z\cos\angle_z
        \quad\text{and}\quad
        \realopt\mu=-\Im z=-\ampl_z\sin\angle_z,
    \]
    which is analogous to \eqref{eq:complex-reconstr-rho>0-lambda-mu}. Using the assumption that  $\realopt\angle=\angle_z$, we can then proceed as in the proof of \cref{lemma:complex-reconstr-rho>0} to derive the estimate \eqref{eq:complex-reconstr-intermediary-d1} with $\ampl_z$ in place of $\alpha$, which for $\beta=1$ reads
    \begin{equation*}
        d_1 \ge
        \ampl_z\left[\frac{\cos\eta_1}{2}\ampl(\realopt\angle-\angle')^2
        -\frac{\cos\eta_2}{4}\ampl'(\angle-\angle')^2
        \right].
    \end{equation*}
    Now we have for any $\zeta>0$ that
    \[
        (\realopt\angle-\angle')^2 =(\angle-\realopt\angle)^2+(\angle-\angle')^2 -2(\angle-\realopt\angle)(\angle-\angle') \ge (1-\zeta)(\angle-\realopt\angle)^2-(\inv\zeta-1)(\angle-\angle')^2
    \]
    and therefore
    \[
        d_1 \ge
        \frac{\cos\eta_1}{2}(1-\zeta)\ampl_z\ampl (\angle-\realopt\angle)^2
        -\ampl_z\left[\frac{\cos\eta_2}{4}\ampl'
            +\frac{\cos\eta_1}{2}(\inv\zeta-1)\ampl
        \right](\angle-\angle')^2.
    \]
    As in the proof of \cref{lemma:complex-reconstr-rho>0}, we also have that $\norm{D_2}^2 = \ampl^2(\realopt\angle-\angle)^2$.
    Now taking $\zeta\in (0,1)$ and $\varepsilon >0$ small enough, we can force $0=\realopt\ampl\leq \ampl\leq \varepsilon$ sufficiently small that $\frac{1}{2}(1-\zeta)\cos\eta_1\ampl_z\ampl > \theta\ampl^2$ for any given $\theta>0$.
    Likewise, we can guarantee
    \[
        \ampl_z\bigl[\frac{1}{4}\ampl'\cos\eta_2 +\frac{\inv\zeta-1}{2}\ampl\cos\eta_1 \bigr] \le L
    \]
    for any $\zeta > 0$ provided $\ampl',\ampl \ge 0$ are small enough.
    We therefore obtain that $d_1 \ge \theta\norm{D_2}^2 - L(\angle-\angle')^2$ and hence the claim.
\end{proof}

\printbibliography

\end{document}